\documentclass[a4paper,12pt]{article}
\usepackage{fancyhdr,graphicx}
\usepackage{amsfonts}
\usepackage{amssymb}
\usepackage{amsthm}
\usepackage{newlfont}
\usepackage{amsmath}
\usepackage[top=2.5cm,bottom=2.5cm,left=2.5cm,right=2.5cm]{geometry}

\usepackage{lineno}

\newtheorem{thm}{Theorem}[section]
\newtheorem{Lemma}[thm]{Lemma}

\usepackage{latexsym,bm}
\title{\vspace{0cm}{Hamiltonian cycles passing through matchings in $k$-ary $n$-cubes \thanks{This work is supported by National Natural Science Foundation of China (grant no. 12061047), and supported by Natural Science Foundation of Jiangxi Province (nos. 20212BAB201027 and 20192BAB211002).}}}

\author{Baolai Liao$^{1}$, Fan Wang$^{1,}$\footnote{Corresponding author.}
\\\small{1. School of Mathematics and Computer Sciences, Nanchang University,}
\\\small{Nanchang, Jiangxi 330000, P. R. China}
\\\small{E-mail addresses: baolai0522ncu@foxmail.com, wangfan@ncu.edu.cn}}

\date{}

\begin{document}


\maketitle

\begin{abstract}
As we all know, the $k$-ary $n$-cube is a highly efficient interconnect network topology structure. It is also a concept of great significance, with a broad range of applications spanning both mathematics and computer science. In this paper, we study the existence of Hamiltonian cycles passing through prescribed matchings in $k$-ary $n$-cubes, and obtain the following result. For $n\geq5$ and $k\geq4$, every matching with at most $4n-20$ edges is contained in a Hamiltonian cycle in the $k$-ary $n$-cube.
\end{abstract}



\section{Introduction}

The Hamiltonian problem constitutes a pivotal subfield within graph theory. Due to the fact that many problems in operations research, computer science, and coding theory can be transformed into Hamiltonian problems, it has attracted widespread attention and research \cite{5,10,11,15,31,32,35}.

The {\it $k$-ary $n$-cube}, denoted by $Q_n^{k}$, where $n\geq1$ and $k\geq2$, is one of the most popular interconnection networks for parallel and distributed systems \cite{20,21,23,25,33,34}. It is a graph consisting of $k^{n}$ vertices, each of which has the form $u=\alpha_1\cdots\alpha_n$, where $0\leq\alpha_i\leq k-1$ for every $i\geq1$. Two vertices $u=\alpha_1\cdots\alpha_n$ and $v=\beta_1\cdots\beta_n$ are adjacent if and only if there exists an integer $j\in\{1,\ldots,n\}$ such that
$\alpha_j=\beta_j\pm1$ (mod $k$) and $\alpha_i=\beta_i$ for every $i\in\{1,\ldots,n\}\setminus\{j\}$. Each vertex in $Q_n^{k}$ has $2n$ neighbors with $k\geq3$, and there is no cycle of length 3 in $Q_n^{k}$ when $k\geq4$.

It is well known that $Q_n^k$ is Hamiltonian \cite{6}. Given an edge set in the $k$-ary $n$-cube, which conditions guarantee the existence of a Hamiltonian cycle in the $k$-ary $n$-cube containing the edge set? Many scholars investigated the problem of embedding Hamiltonian cycles and paths in the hypercube (2-ary $n$-cube) with prescribed edges \cite{10,11,15,31}. With the deepening of research on hypercube, the corresponding conclusion has gradually been extended to $k$-ary $n$-cube \cite{24,26,27,29,32,35}.

A forest is {\it linear} if each component of it is a path. Wang et al. \cite{26} and Stewart \cite{24} independently investigated the case where $k\geq3$, and they both obtained the following result.

\begin{thm}\cite{24,26}\label{forest} For $n\geq2$ and $k\geq3$, if $F$ is a linear forest in $Q_n^{k}$ with $|F|\leq2n-1$, then $F$ is contained in a Hamiltonian cycle in $Q_n^{k}$.
\end{thm}

Wang et al. \cite{29} considered the problem of embedding Hamiltonian cycles in the $k$-ary $n$-cubes with small matchings.

\begin{thm}\cite{29}\label{maintheorem} For $n\geq2$ and $k\geq3$, if $M$ is a matching in $Q_n^{k}$ with $|M|\leq3n-8$, then $M$ is contained in a Hamiltonian cycle in $Q_n^{k}$.
\end{thm}

In this paper, we consider the problem of embedding Hamiltonian cycles in the $k$-ary $n$-cubes with larger matchings, and obtain the following main result.

\begin{thm}\label{mosttheorem} For $n\geq5$ and $k\geq4$, if $M$ is a matching in $Q_n^{k}$ with $|M|\leq4n-20$, then $M$ is contained in a Hamiltonian cycle in $Q_n^{k}$.
\end{thm}

\section{Basic definitions and results}
The vertex set and edge set of a graph $G$ are denoted by $V(G)$ and $E(G)$, respectively. Let $F\subseteq E(G)$ and $U\subseteq V(G)$, and let $H,H'$ be two subgraphs of $G$. The notation $V(F)$ represents the set of all vertices incident with edges in $F$. Denote $G-F$ as the resulting graph by deleting from $G$ the edges in $F$, and denote $G-U$ as the resulting graph by deleting from $G$ the vertices in $U$ together with all the edges incident with $U$. Let $H+H'$ represent the graph with the vertex set $V(H)\cup V(H')$ and edge set $E(H)\cup E(H')$, and let $H+F$ denote the graph with the vertex set $V(H)\cup V(F)$ and edge set $E(H)\cup F$.

In a graph $G$, the distance between vertices $u$ and $x$ is denoted by $d_G(u,x)$, and the distance of a vertex $u$ and an edge $xy$ is defined by $d_G(u,xy)=\min\{d_G(u,x),d_G(u,y)\}$. For a set $M\subseteq E(G)$, we say that a subgraph $H$ of $G$ {\it passes through} $M$, if $E(M)\subseteq E(H)$ and $V(M)\subseteq V(H)$.

A {\it $u,v$-path} is a path with endpoints $u$ and $v$, denoted by $P_{u,v}$ when we specify a particular such path. A spanning subgraph of $G$ is called a {\it spanning $m$-path} if it consists of $m$ vertex-disjoint paths. A spanning 1-path is a spanning or Hamiltonian path. Let $P_{x,y}=x,\ldots,v_i,\ldots,v_j,\ldots,y$ be a path. The notation $P_{x,y}[v_i,v_j]$ denotes the subpath $v_i,\ldots,v_j$ of $P_{x,y}$ joining $v_i$ and $v_j$. For $u\in V(P_{x,y})\setminus\{x\}$, a neighbor $u'$ of $u$ on $P_{x,y}$ is called {\it closer to} $x$ than $u$ if $u'$ lies on the subpath $P_{x,y}[x,u]$.

Note that $Q_n^{k}$ is bipartite if and only if $k$ is even. When $k$ is even, the {\it parity} of a vertex $u=\alpha_1\cdots\alpha_n$ in $Q_n^{k}$ is defined by $p(u)=\sum_{i=1}^{n}\alpha_i$ (mod 2), and a set $\{u_{1},u_{2},\ldots,u_{2n-1},u_{2n}\}$ of distinct vertices is $balanced$ in $Q_n^{k}$ if the number of odd vertices equals to the number of even vertices. For even $k$, $p(u)\neq p(v)$ if $uv\in E(Q_n^{k})$, and the condition ``$p(u)\neq p(v)$'' is necessary for the existence of a spanning $u,v$-path in $Q_n^{k}$.

Next, let us introduce several lemmas.

\begin{Lemma}\cite{33}\label{hamilpath} For $n\geq2$ and odd $k\geq3$, let $U\subseteq V(Q_n^{k})$ with $|U|\leq2n-3$. If $x,y\in V(Q_n^{k})\setminus U$ are distinct, then there is a spanning $x,y$-path in $Q_n^{k}-U$.
\end{Lemma}

\begin{Lemma}\cite{25}\label{hamilpatheven} For $n\geq2$ and even $k\geq4$, if $x,y$ are vertices of different parities in $Q_n^{k}$, then there is a spanning $x,y$-path in $Q_n^{k}$.
\end{Lemma}

\begin{Lemma}\cite{29}\label{delete} For $n\geq3$ and even $k\geq4$, if $u,x,y$ are distinct vertices in $Q_n^{k}$ such that $p(u)\neq p(x)=p(y)$, then there is a spanning $x,y$-path in $Q_n^{k}-u$.
\end{Lemma}

\begin{Lemma}\cite{29}\label{hamilpath3} For $n\geq3$ and $k\geq3$, let $x,y$ be distinct vertices in $Q_n^{k}$ and whenever $k$ is even, $p(x)\neq p(y)$. If $uv\in E(Q_n^{k})$ and $\{u,v\}\neq\{x,y\}$, then there is a spanning $x,y$-path passing through $uv$ in $Q_n^{k}$.
\end{Lemma}

\begin{Lemma}\cite{29}\label{pathpartition2} For $n\geq3$ and $k\geq3$, let $M$ be a matching in $Q_n^{k}$ with $|M|\leq\max\{0,2n-7\}$. If $uu'$ and $vv'$ are two disjoint edges in $Q_n^{k}$ such that $\{u,v\}\cap V(M)=\emptyset$ and $u'v'\notin M$, then there is a spanning 2-path $P_{u,u'}+P_{v,v'}$ passing through $M$ in $Q_n^{k}$.
\end{Lemma}

Let $d\in\{1,\ldots,n\}$. An edge in $Q_n^{k}$ is an {\it $d$-edge} if its endpoints differ in the $d$th position. Let $E_d$ denote the set of all $d$-edges in $Q_n^{k}$. Let $Q_{n-1}^{k,d}[i]$ be the $(n-1)$-dimensional subcube of $Q_n^{k}$ induced by all the vertices with the $d$th position being $i$ where $i\in\{0,1,\ldots,k-1\}$, abbreviated as $Q[i]$. Obviously, $Q_n^{k}-E_d=Q[0]+\cdots+Q[k-1]$. We say that $Q_n^{k}$ splits into $(n-1)$-dimensional subcubes $Q[0],\ldots,Q[k-1]$ by $E_d$. For every vertex $u_i=\alpha_1\cdots\alpha_{d-1}i$ $\alpha_{d+1}\cdots\alpha_n$ in $Q[i]$, let $u_j$ represent the vertex $\alpha_1\cdots\alpha_{d-1}j$ $\alpha_{d+1}\cdots\alpha_n$ in $Q[j]$ where $0\leq j\leq k-1$.

In the following definitions, they are all defined under the premise of modulo $k$. Let $p,q$ be two integers with $0\leq p\leq q\leq k-1$. Denote $E_d(p,p+1)$ as the set of all $d$-edges joining $Q[p]$ and $Q[p+1]$. Let $E_d(p,q)=\bigcup_{i=p}^{q-1}E_d(i,i+1)$ when $p<q$, and let $E_d(p,q)=\emptyset$ when $p=q$. Let $Q_{n-1}^{k,d}[p,q]$, abbreviated as $Q[p,q]$, represent the graph with the vertex set $\bigcup_{i=p}^{q}V(Q[i])$ and edge set $(\bigcup_{i=p}^{q}E(Q[i]))\cup E_d(p,q)$. Then $Q[p,p]=Q[p]$ and $Q[0,k-1]=Q_n^{k}-E_d(k-1,0)$. Given $M\subseteq E(Q_n^{k})$, let $M[p,q]=M\cap E(Q[p,q])$, and we briefly write $M[p,p]$ as $M_p$.

\begin{Lemma}\cite{29}\label{hamilpath1} For $n\geq3$ and $k\geq3$, if $x,y$ are distinct vertices in $Q_{n-1}^{k,d}[p,q]$ and $p(x)\neq p(y)$ when $k$ is even, then there is a spanning path $P_{x,y}$ in $Q[p,q]$ such that $E(P_{x,y})\cap E(Q[q])$ forms a spanning path or 2-path in $Q[q]$.
\end{Lemma}

\begin{Lemma}\cite{29}\label{hamilpath2} For $n\geq3$ and $k\geq3$, let $M$ be a matching in $Q_{n-1}^{k,d}[p,q]$ such that $|M_i|\leq2n-4$ for every $i\in\{p,\ldots,q\}$ and $|M\cap E_d(i,i+1)|\leq1$ for every $i\in\{p,\ldots,q-1\}$ when $q>p$. If $xy\in E(Q[p])\setminus M$, then there is a spanning $x,y$-path $P_{x,y}$ passing through $M$ in $Q[p,q]$ such that $E(P_{x,y})\cap E(Q[q])$ forms a spanning path in $Q[q]$ and $E(P_{x,y})\cap E(Q[p])$ forms a spanning 2-path in $Q[p]$.
\end{Lemma}

\begin{Lemma}\cite{29}\label{pathpartition3} For $n\geq4$ and $k\geq3$, let $x,y,u,v$ be distinct vertices in $Q_{n-1}^{k,d}[p,q]$ and whenever $k$ is even, $p(x)\neq p(y)$ and $p(u)\neq p(v)$. If $x,y,u,v\in V(Q[p])$ or $x,y\in V(Q[p])$ and $u,v\in V(Q[q])$, then there is a spanning 2-path $P_{x,y}+P_{u,v}$ in $Q[p,q]$.
\end{Lemma}

\section{Spanning $m$-paths in $Q_n^{k}$}
In this section, we testify some Lemmas about spanning $m$-paths which are applied in the constructions of Hamiltonian cycles in our result. In the following lemmas, we always assume $k\geq4$.

\begin{Lemma}\label{hamilpathone} For $n\geq4$, let $x,y$ be distinct vertices in $Q_{n-1}^{k,d}[p,q]$ and whenever $k$ is even, $p(x)\neq p(y)$. Let $M$ be a matching in $Q_{n-1}^{k,d}[p,q]$ such that $|M|\leq1$ and $xy\notin M$. If $x,y\in V(Q[p])$ or $x\in V(Q[p])$ and $y\in V(Q[q])$, then there is a spanning $x,y$-path passing through $M$ in $Q[p,q]$ such that $E(P_{x,y})\cap E(Q[q])$ forms a spanning path or 2-path in $Q[q]$.
\end{Lemma}

\begin{proof} Note that $0\leq p\leq q\leq k-1$. If $M=\emptyset$, the result holds by Lemma \ref{hamilpath1}. If $|M|=1$, we prove the result by induction on $q-p$. When $q=p$, it holds by Lemma \ref{hamilpath3}. Now consider $q>p$, and suppose that the lemma holds for all $Q[p',q']$ with $q'-p'<q-p$.

\textbf{Case 1.} $x,y\in V(Q[p])$.

By Lemma \ref{hamilpath3} there is a spanning path $P_{x,y}'$ passing through $M_p$ in $Q[p]$. If $M\cap E_d(p,p+1)=\emptyset$, since $|E(P_{x,y}')|-|M|\geq k^{n-1}-1>1$, choose an edge $s_{p}s_{p}'\in E(P_{x,y}')\setminus M$ such that $ s_{p+1}s_{p+1}'\notin M$. If $M\cap E_d(p,p+1)\neq\emptyset$, let $M=\{s_{p}s_{p+1}\}$ and $s_{p}'$ be a neighbor of $s_{p}$ on $P_{x,y}'$. Since $|M[p+1,q]|\leq1$, by Lemma \ref{hamilpath2} there is a spanning path $P_{s_{p+1},s_{p+1}'}$ passing through $M[p+1,q]$ in $Q[p+1,q]$. Hence, $P_{x,y}=P_{x,y}'+P_{s_{p+1},s_{p+1}'}+\{s_{p}s_{p+1},s_{p}'s_{p+1}'\}-\{s_{p}s_{p}'\}$ is the desired spanning path in $Q[p,q]$.

\textbf{Case 2.} $x\in V(Q[p])$ and $y\in V(Q[q])$.

If there exists an integer $j\in\{p,\ldots,q-1\}$ such that $M\cap E_d(j,j+1)=\emptyset$, let $s_j\in V(Q[j])$ such that $s_{j}\neq x$, $s_{j+1}\neq y$ and $p(s_{j+1})=p(x)\neq p(y)=p(s_{j})$ when $k$ is even. By the induction hypothesis there exist spanning paths $P_{x,s_{j}}$ passing through $M[p,j]$ in $Q[p,j]$ and $P_{s_{j+1},y}$ passing through $M[j+1,q]$ in $Q[j+1,q]$, respectively. Hence, $P_{x,y}=P_{x,s_{j}}+P_{s_{j+1},y}+\{s_{j}s_{j+1}\}$ is the desired spanning path in $Q[p,q]$.

Otherwise, $p+1=q$ and $|M\cap E_d(p,p+1)|=1$. Let $M\cap E_d(p,p+1)=\{s_ps_{p+1}\}$. If $y\neq s_{p+1}$, since $s_{p}$ has $2(n-1)\geq6$ neighbors in $Q[p]$, we may choose a neighbor $s_{p}'$ of $s_{p}$ in $Q[p]$ such that $s_{p+1}'\neq y$. Since $|V(Q[p+1])|=k^{n-1}\geq 7$, we may choose a vertex $t_{p+1}\in V(Q[p+1])\setminus\{y,s_{p+1},s_{p+1}'\}$ such that $t_{p}\neq x$, $xt_{p}\neq s_{p}s_{p}'$ and $p(x)\neq p(t_{p})$ when $k$ is even. Then $s_{p+1},s_{p+1}',y,t_{p+1}$ are distinct and $p(s_{p+1})\neq p(s_{p+1}')$, $p(y)\neq p(t_{p+1})$ when $k$ is even. By the induction hypothesis there is a spanning path $P_{x,t_{p}}$ passing through $s_{p}s_{p}'$ in $Q[p]$, and by Lemma \ref{pathpartition3} there is a spanning 2-path $P_{s_{p+1},s_{p+1}'}+P_{y,t_{p+1}}$ in $Q[p+1]$. Hence, $P_{x,y}=P_{x,t_{p}}+P_{s_{p+1},s_{p+1}'}+P_{y,t_{p+1}}+\{s_{p}s_{p+1},s_{p}'s_{p+1}',t_{p}t_{p+1}\}-\{s_{p}s_{p}'\}$ is the desired spanning path in $Q[p,q]$. If $y=s_{p+1}$, now $x\neq s_{p}$. Choose a vertex $t_{p}\in V(Q[p])\setminus \{x,s_{p}\}$ such that $p(t_{p})\neq p(x)$ when $k$ is even. So $t_{p+1}\neq y$. By Lemma \ref{hamilpath1} there is a spanning path $P_{x,t_{p}}$ in $Q[p]$. Denote by $s_{p}'$ the neighbor of $s_{p}$ on $P_{x,t_{p}}$ that is closer to $x$. Then $y,s_{p+1}',t_{p+1}$ are distinct and $p(y)\neq p(s_{p+1}')=p(t_{p+1})$ when $k$ is even. By Lemma \ref{hamilpath} or Lemma \ref{delete} there is a spanning path $P_{s_{p+1}',t_{p+1}}$ in $Q[p+1]-y$. Now $P_{x,t_{p}}+P_{s_{p+1}',t_{p+1}}+\{s_ps_{p+1},s_p's_{p+1}',t_pt_{p+1}\}-\{s_{p}s_{p}'\}$ is the desired spanning path in $Q[p,q]$.
\end{proof}

\begin{Lemma}\label{pathpartition7} For $n\geq4$, let $x,y,u,v$ be distinct vertices in $Q_{n-1}^{k,d}[p,q]$ where $p<q$ and whenever $k$ is even, $p(x)\neq p(y)$ and $p(u)\neq p(v)$. If $x,y,u,v\in V(Q[p])$ and $M$ is a matching in $Q_{n-1}^{k,d}[p,q]$ such that $|M|\leq1$ and $|\{x,y,u,v\}\cap V(M)|\leq1$, then there is a spanning 2-path $P_{x,y}+P_{u,v}$ passing through $M$ in $Q[p,q]$.
\end{Lemma}

\begin{proof} If $|M_p|=1$, since $|\{x,y,u,v\}\cap V(M)|\leq1$, we may assume $\{u,v\}\cap V(M)=\emptyset$. By Lemma \ref{hamilpath3} there exists a spanning path $P_{x,y}'$ passing through $M_p$ in $Q[p]$. Let $u_{p}',v_{p}'$ be neighbors of $u,v$ on $P_{x,y}'$ such that $u_{p}'\notin V(P_{x,y}[u,v])$ and $v_{p}'\notin V(P_{x,y}[u,v])$. Note that $\{uu_{p}',vv_{p}'\}\cap M=\emptyset$, and $u_{p}'\neq v_{p}'$, and $p(u_{p}')\neq p(v_{p}')$ when $k$ is even. By Lemma \ref{hamilpath1} there is a spanning path $P_{u_{p+1}',v_{p+1}'}$ in $Q[p+1,q]$. Let $P_{x,y}=P_{x,y}'-V(P_{x,y}'[u,v])+P_{u_{p+1}',v_{p+1}'}+\{u_{p}'u_{p+1}',v_{p}'v_{p+1}'\}$ and $P_{u,v}=P_{x,y}'[u,v]$. Now $P_{x,y}+P_{u,v}$ is the desired spanning 2-path in $Q[p,q]$.

Otherwise, $|M_p|=0$. Then by Lemma \ref{pathpartition3} there is a spanning 2-path $P_{x,y}'+P_{u,v}'$ in $Q[p]$. If $M\cap E_d(p,p+1)= \emptyset$, let $s_{p}s_{p}'\in E(P_{x,y}'+P_{u,v}')$ such that $s_{p+1}s_{p+1}'\notin M_{p+1}$. If $M\cap E_d(p,p+1)\neq \emptyset$, let $M\cap E_d(p,p+1)=\{s_{p}s_{p+1}\}$ and $s_{p}s_{p}'$ be an edge on $P_{x,y}'+P_{u,v}'$. Note that $|M[p+1,q]|\leq1$, by Lemma \ref{hamilpath2} there exists a spanning path $P_{s_{p+1},s_{p+1}'}$ passing through $M[p+1,q]$ in $Q[p+1,q]$. Without loss of generality, we may assume $s_{p}s_{p}'\in E(P_{x,y}')$. Let $P_{x,y}=P_{x,y}'+P_{s_{p+1},s_{p+1}'}+\{s_{p}s_{p+1},s_{p}'s_{p+1}'\}-\{s_{p}s_{p}'\}$ and $P_{u,v}=P_{u,v}'$. Now $P_{x,y}+P_{u,v}$ is the desired spanning 2-path in $Q[p,q]$.
\end{proof}

\begin{Lemma}\label{hamilpathtwo} For $n\geq4$, let $x,y$ be distinct vertices in $Q_n^{k}$ and whenever $k$ is even, $p(x)\neq p(y)$. If $M$ is a matching in $Q_n^{k}$ with $|M|\leq2$ and $xy\notin M$, then there is a spanning $x,y$-path passing through $M$ in $Q_n^{k}$.
\end{Lemma}

\begin{proof} When $|M|\leq1$, the result holds by Lemma \ref{hamilpath3}. Next consider $|M|=2$. Since $n\geq4$, we may split $Q_n^{k}$ into subcubes $Q[0],\ldots,Q[k-1]$ such that $M\cap E_d=\emptyset$. By symmetry we may assume $x\in V(Q[0])$ and $y\in V(Q[i])$, where $0\leq i\leq k-1$.

\textbf{Case 1.} $i=0$.

If $|M_0|\leq1$, by Lemma \ref{hamilpath3} there exists a spanning path $P_{x,y}'$ passing through $M_0$ in $Q[0]$. Let $s_{0}s_{0}'\in E(P_{x,y}')\setminus M_0$ such that $s_{1}s_{1}'\notin M$. Note that $|M_i|\leq2$ for every $i\in\{1,\ldots,k-1\}$ and $M\cap E_d=\emptyset$. By Lemma \ref{hamilpath2} there exists a spanning path $P_{s_{1},s_{1}'}$ passing through $M[1,k-1]$ in $Q[1,k-1]$. Now $P_{x,y}=P_{x,y}'+P_{s_{1},s_{1}'}+\{s_{0}s_{1},s_{0}'s_{1}'\}-\{s_{0}s_{0}'\}$ is the desired spanning path in $Q_n^{k}$.

Otherwise, $|M_0|=|M|=2$. Let $uv\in M_0$. By Lemma \ref{hamilpathone} there exists a spanning path $P_{x,y}'$ passing through $M_0\setminus\{uv\}$ in $Q[0]$. If $uv\in E(P_{x,y}')$, choose an edge $s_{0}s_{0}'\in E(P_{x,y}')\setminus M_0$. By Lemma \ref{hamilpath2} there exists a spanning path $P_{s_{1},s_{1}'}$ in $Q[1,k-1]$. Let $P_{x,y}=P_{x,y}'+P_{s_{1},s_{1}'}+\{s_{0}s_{1},s_{0}'s_{1}'\}-\{s_{0}s_{0}'\}$. If $uv\notin E(P_{x,y}')$, since $uv\neq xy$, we may choose neighbors $u_{0}',v_{0}'$ of $u,v$ on $P_{x,y}'$ such that exactly one of $\{u_{0}',v_{0}'\}$ lies on $P_{x,y}'[u,v]$. Since $uv\in M_0$, we have $\{uu_{0}',vv_{0}'\}\cap M=\emptyset$. Now $u_{1}',v_{1}'$ are distinct and $p(u_{1}')\neq p(v_{1}')$ when $k$ is even. By Lemma \ref{hamilpath1} there exists a spanning path $P_{u_{1}',v_{1}'}$ in $Q[1,k-1]$. Let $P_{x,y}=P_{x,y}'+P_{u_{1}',v_{1}'}+\{uv,u_{0}'u_{1}',v_{0}'v_{1}'\}-\{uu_{0}',vv_{0}'\}$. Hence, $P_{x,y}$ is the desired spanning path in $Q_n^{k}$.

\textbf{Case 2.} $1\leq i\leq k-1$.

Since $|M[0,i-1]|+|M[i,k-1]|=|M|=2$, without loss of generality, assume $|M[i,k-1]|\leq1$. Choose a neighbor $x_{0}'$  of $x$ in $Q[0]$ such that $x_{k-1}'\neq y$ and $\{xx_{0}',yx_{k-1}'\}\cap M=\emptyset$. Note that $p(x_{k-1}')=p(x)\neq p(y)$ when $k$ is even. Since $M\cap E_d=\emptyset$, by Lemma \ref{hamilpath2} and Lemma \ref{hamilpathone} there exist spanning paths $P_{x,x_{0}'}$ passing through $M[0,i-1]$ in $Q[0,i-1]$ and $P_{y,x_{k-1}'}$ passing through $M[i,k-1]$ in $Q[i,k-1]$, respectively. Now $P_{x,y}=P_{x,x_{0}'}+P_{y,x_{k-1}'}+\{x_{0}'x_{k-1}'\}$ is the desired spanning path in $Q_n^{k}$.
\end{proof}

\begin{Lemma}\label{pathpartition8} For $n\geq5$, let $x,y$ be distinct vertices in $Q_{n-1}^{k,d}[p,q]$ such that $x,y\in V(Q[p])$ and whenever $k$ is even, $p(x)\neq p(y)$. If $M$ is a matching in $Q_{n-1}^{k,d}[p,q]$ such that $|M|\leq2$ and $xy\notin M$, then there is a spanning $x,y$-path passing through $M$ in $Q_{n-1}^{k,d}[p,q]$.
\end{Lemma}

\begin{proof} Note that $0\leq p\leq q\leq k-1$. We prove the lemma by induction on $q-p$. When $p=q$, the result holds by Lemma \ref{hamilpathtwo}. Next consider $p<q$, and suppose that the lemma holds for all $Q[p',q']$ with $q'-p'<q-p$.

By the induction hypothesis there exists a spanning path $P_{x,y}'$ passing through $M_p$ in $Q[p]$. If $M\cap E_d(p,p+1)=\emptyset$, since $|E(P_{x,y}')|-|M_{p}|-|M_{p+1}|\geq k^{n-1}-1-2>1$, we may choose an edge $s_{p}s_{p}'\in E(P_{x,y}')\setminus M_p$ such that $s_{p+1}s_{p+1}'\notin M_{p+1}$. If $|M\cap E_d(p,p+1)|=1$, let $M\cap E_d(p,p+1)=\{s_{p}s_{p+1}\}$ and $s_{p}'$ be a neighbor of $s_{p}$ on $P_{x,y}'$, then $\{s_{p}s_{p}',s_{p+1}s_{p+1}'\}\cap M=\emptyset$. If $|M\cap E_d(p,p+1)|=2$, let $M=\{s_{p}s_{p+1},t_{p}t_{p+1}\}$ and furthermore, if $d_{P_{x,y}'}(s_{p},t_{p})=1$, then let $t_{p}=s_{p}'$. In the above three cases, $s_{p+1},s_{p+1}'$ are distinct and whenever $k$ is even, $p(s_{p+1})\neq p(s_{p+1}')$. By the induction hypothesis there is a spanning path passing through $M[p+1,q]$ in $Q[p+1,q]$. Hence, $P_{x,y}=P_{x,y}'+P_{s_{p+1},s_{p+1}'}+\{s_{p}s_{p+1},s_{p}'s_{p+1}'\}-\{s_{p}s_{p}'\}$ is the desired spanning path in $Q_{n-1}^{k,d}[p,q]$. It remains to consider the case $M=\{s_{p}s_{p+1},t_{p}t_{p+1}\}$ and $d_{P_{x,y}'}(s_{p},t_{p})>1$. Choose neighbors $s_p',t_p'$ of $s_{p},t_{p}$ on $P_{x,y}'$, respectively, such that
$s_p'\neq t_p'$. By Lemma \ref{pathpartition3} there is a spanning 2-path $P_{s_{p+1},s_{p+1}'}+P_{t_{p+1},t_{p+1}'}$ in $Q[p+1,q]$. Now $P_{x,y}=P_{x,y}'+P_{s_{p+1},s_{p+1}'}+P_{t_{p+1},t_{p+1}'}+\{s_{p}s_{p+1},$
$s_{p}'s_{p+1}',t_{p}t_{p+1},t_{p}'t_{p+1}'\}-\{s_{p}s_{p}',t_{p}t_{p}'\}$ is the desired spanning path in $Q_{n-1}^{k,d}[p,q]$.
\end{proof}

\begin{Lemma}\label{pathpartition9} For $n\geq5$, let $x,y,u,v$ be distinct vertices in $Q_n^{k}$, and $p(x)\neq p(y)$ whenever $k$ is even. If $M$ is a matching in $Q_n^{k}$ and $uv\in E(Q_n^{k})$ such that $|M|\leq1$, and $V(M)\cap\{u,v\}=\emptyset$, and $xy\notin M$, then there is a spanning $x,y$-path passing through $M$ in $Q_n^{k}-\{u,v\}$.
\end{Lemma}

\begin{proof} Since $|M|\leq1$, we may split $Q_n^{k}$ into subcubes $Q[0],\ldots,Q[k-1]$ such that $M\cap E_d=\emptyset$ and $u,v$ lie in the same subcube. By symmetry we may assume $uv\in E(Q[0])$. Let $x\in V(Q[i])$ and $y\in V(Q[j])$.

\textbf{Case 1.} $i=j=0$.

Since $M_0\cup\{uv\}$ is a matching with $|M_0\cup\{uv\}|\leq2$ and $xy\notin M_0\cup\{uv\}$, by Lemma \ref{pathpartition8} there is a spanning path $P_{x,y}'$ passing through $M_0\cup\{uv\}$ in $Q[0]$. Let $u_{0}',v_{0}'$ be neighbors of $u,v$ on $P_{x,y}'$ such that $u_{0}'\neq v$ and $v_{0}'\neq u$. Now $u_{1}',v_{k-1}'$ are distinct and $p(u_{1}')=p(u)\neq p(v)=p(v_{k-1}')$ when $k$ is even. Since $|M|\leq1$ and $M\cap E_d=\emptyset$, by Lemma \ref{hamilpathone} there is a spanning path $P_{u_{1}',v_{k-1}'}'$ passing through $M[1,k-1]$ in $Q[1,k-1]$. Hence, $P_{x,y}=P_{x,y}'+P_{u_{1}',v_{k-1}'}+\{u_{0}'u_{1}',v_{0}'v_{k-1}'\}-\{u,v\}$ is the desired spanning path in $Q_n^{k}-\{u,v\}$.

\textbf{Case 2.} $i=0, j\neq0$ or $i\neq0, j=0$. By symmetry assume $i=0, j\neq0$.

Since $k\geq4$, by symmetry assume $j\neq 1$. When $k$ is even, since $p(u)\neq p(v)$, exactly one of the conditions $p(x)\neq p(u)$ and $p(x)\neq p(v)$ holds true. By symmetry assume $p(x)\neq p(u)$. Since $M_0\cup\{uv\}$ is a matching with $|M_0\cup\{uv\}|\leq2$ and $xu\notin M_0\cup\{uv\}$, by Lemma \ref{pathpartition8} there is a spanning path $P_{x,u}$ passing through $M_0\cup\{uv\}$ in $Q[0]$. Choose the neighbor $v_{0}'$ of $v$ on $P_{x,u}$ satisfaying $v_{0}'\neq u$. Note that $v_{1}'\neq y$ and $p(v_{1}')\neq p(y)$ when $k$ is even. Since $M\cap E_d=\emptyset$, we have $v_{1}'y\notin M$. By Lemma \ref{hamilpathone} there is a spanning path $P_{v_{1}',y}$ passing through $M[1,j]$ in $Q[1,j]$ such that $E(P_{v_{1}',y})\cap E(Q[j])$ forms a spanning path or 2-path in $Q[j]$. If $j=k-1$, let $P_{x,y}=P_{x,u}[x,v_{0}']+P_{v_{1}',y}+\{v_{0}'v_{1}'\}$. If $j<k-1$, we may choose an edge $r_{j}r_{j}'\in E(P_{v_{1}',y})\cap E(Q[j])\setminus M_{j}$ such that $r_{j+1}r_{j+1}'\notin M_{j+1}$. By Lemma \ref{hamilpath2} there exists a spanning path $P_{r_{j+1},r_{j+1}'}$ passing through $M[j+1,k-1]$ in $Q[j+1,k-1]$. Let $P_{x,y}=P_{x,u}[x,v_{0}']+P_{v_{1}',y}+P_{r_{j+1},r_{j+1}'}+\{v_{0}'v_{1}',r_{j}r_{j+1},r_{j}'r_{j+1}'\}-\{r_{j}r_{j}'\}$. Now $P_{x,y}$ is the desired spanning path in $Q_n^{k}-\{u,v\}$.

\textbf{Case 3.} $i\neq0$ and $j\neq0$, by symmetry assume $i\leq j$.

Without loss of generality, assume $p(x)=p(v)\neq p(u)=p(y)$ when $k$ is even. Since $u,v$ both have $2(n-1)\geq8$ neighbors in $Q[0]$, choose neighbors $u_{0}',v_{0}'$ of $u,v$ in $Q[0]$ such that $u,u_{0}',v,v_{0}'$ are distinct, and $\{u_{0}',v_{0}',u_{1}',v_{k-1}'\}\cap V(M)=\emptyset$, and $u_{1}'\neq x, v_{k-1}'\neq y$. Thus $M_0\cup\{uv,uu_{0}',vv_{0}'\}$ is a linear forest in $Q[0]$. By Theorem \ref{forest} there is a Hamiltonian cycle $C_0$ containing $M_0\cup\{uv,uu_{0}',vv_{0}'\}$ in $Q[0]$.

If $i< j$, since $u_{1}',x,v_{k-1}',y$ are distinct and whenever $k$ is even, $p(x)\neq p(u)=p(u_{1}')$ and $p(y)\neq p(v)=p(v_{k-1}')$, by Lemma \ref{hamilpathone} there exist spanning paths $P_{x,u_{1}'}'$ passing through $M[1,i]$ in $Q[1,i]$ such that $E(P_{x,u_{1}'}')\cap E(Q[i])$ forms a spanning path or 2-path in $Q[i]$ and $P_{y,v_{k-1}'}$ passing through $M[j,k-1]$ in $Q[j,k-1]$, respectively. When $i+1=j$, let $P_{x,u_{1}'}=P_{x,u_{1}'}'$. When $i+1<j$, choose an edge $t_{i}t_{i}'\in E(P_{x,u_{1}'}')\cap E(Q[i])\setminus M_{i}$ such that $t_{i+1}t_{i+1}'\notin M_{i+1}$. By Lemma \ref{hamilpath2} there exists a spanning path $P_{t_{i+1},t_{i+1}'}$ passing through $M[i+1,j-1]$ in $Q[i+1,j-1]$. Let $P_{x,u_{1}'}=P_{x,u_{1}'}'+P_{t_{i+1},t_{i+1}'}+\{t_{i}t_{i+1},t_{i}'t_{i+1}'\}-\{t_{i}t_{i}'\}$. Hence, $P_{x,y}=C_0+P_{x,u_{1}'}+P_{y,v_{k-1}'}+\{u_{0}'u_{1}',v_{0}'v_{k-1}'\}-\{u,v\}$ is the desired spanning path in $Q_n^{k}-\{u,v\}$.

If $i=j$, by symmetry assume $i=j\neq k-1$. Since $|E(C_0)|-|M_0\cup\{uv,uu_{0}',vv_{0}'\}|\geq k^{n-1}-4>4$, we may choose $s_{0}s_{0}'\in E(C_0)\setminus(M_0\cup\{uv,uu_{0}',vv_{0}'\})$ such that $xy\neq s_{1}s_{1}'$ and $\{s_{1},s_{1}'\}\cap V(M_{1})=\emptyset$. Thus, $M[1,i]\cup\{s_{1}s_{1}'\}$ is a matching of size no more than $2$. By Lemma \ref{pathpartition8} there exist spanning paths $P_{x,y}'$ passing through $M[1,i]\cup\{s_{1}s_{1}'\}$ in $Q[1,i]$ and $P_{u_{k-1}',v_{k-1}'}$ passing through $M[i+1,k-1]$ in $Q[i+1,k-1]$, respectively. Now $P_{x,y}=C_0+P_{x,y}'+P_{u_{k-1}',v_{k-1}'}+\{u_{0}'u_{k-1}',v_{0}'v_{k-1}',s_{0}s_{1},s_{0}'s_{1}'\}-\{s_{0}s_{0}',s_{1}s_{1}'\}-\{u,v\}$ is the desired spanning path in $Q_n^{k}-\{u,v\}$.
\end{proof}

\begin{Lemma}\label{pathpartition5} For $n\geq4$ and even $k\geq4$, if $x,y,u,v$ are distinct vertices in $Q_n^{k}$ such that $p(x)= p(y)\neq p(u)= p(v)$, then there is a spanning 2-path $P_{x,y}+P_{u,v}$ in $Q_n^{k}$.
\end{Lemma}

\begin{proof} Since $x,y$ are distinct, we may split $Q_n^{k}$ into subcubes $Q[0],\ldots,Q[k-1]$ such that $x$ and $y$ lie in different subcubes. By symmetry we may assume $x\in V(Q[0])$ and $y\in V(Q[m])$, where $m>0$. Let $u\in V(Q[i])$ and $v\in V(Q[j])$.

\textbf{Case 1.} $i=j$. By symmetry we may assume $0\leq i<m$.

If $i=0$, since $x,u,v$ are distinct vertices and $p(x)\neq p(u)= p(v)$, by Lemma \ref{delete} there is a spanning path $P_{u,v}$ in $Q[0]-x$. Since $k\geq4$, by symmetry we may assume $m\neq1$. Hence, $x_1\neq y$ and $p(x)= p(y)\neq p(x_1)$. By Lemma \ref{hamilpath1} there is a spanning path $P_{x_1,y}$ in $Q[1,k-1]$. Let $P_{x,y}=P_{x_1,y}+\{xx_1\}$.

If $i>0$, then let $s_i\in V(Q[i])$ such that $p(s_i)\neq p(u)= p(v)$. So $p(s_{i-1})=p(s_{i+1})\neq p(x)=p(y)$. By Lemma \ref{delete} there is a spanning path $P_{u,v}$ in $Q[i]-s_i$. By Lemma \ref{hamilpath1} there exist spanning paths $P_{x,s_{i-1}}$ in $Q[0,i-1]$ and $P_{y,s_{i+1}}$ in $Q[i+1,k-1]$. Let $P_{x,y}=P_{x,s_{i-1}}+P_{y,s_{i+1}}+\{s_{i-1}s_i,s_is_{i+1}\}$.

In the above two cases, $P_{x,y}+P_{u,v}$ is the desired spanning 2-path in $Q_n^{k}$.

\textbf{Case 2.} $i\neq j$. By symmetry we may assume $0\leq i< m$ and $i<j$.

If $j<m$, then by Lemma \ref{hamilpath1} there exists a spanning path $P_{x,u}$ in $Q[0,j-1]$ such that $E(P_{x,u})\cap E(Q[j-1])$ forms a spanning path or 2-path in $Q[j-1]$. Let $s_{j-1}r_{j-1}\in E(P_{x,u})\cap E(Q[j-1])$ where $s_{j-1}$ is closer to $x$ on $P_{x,u}$ than $r_{j-1}$. Moreover, since $|E(P_{x,u})\cap E(Q[j-1])|\geq k^{n-1}-2>5$, we may choose $s_{j-1}r_{j-1}$ such that $v\notin\{s_j,r_j\}$, $y\neq s_{j+1}$ and $p(s_{j-1})= p(v)$. Then $p(s_j)\neq p(r_j)= p(v)$ and $p(s_{j+1})\neq p(y)$. By Lemma \ref{delete} and Lemma \ref{hamilpath1} there exist spanning paths $P_{r_{j},v}$ in $Q[j]-s_j$ and $P_{s_{j+1},y}$ in $Q[j+1,k-1]$. Let $P_{x,y}=P_{x,u}[x,s_{j-1}]+P_{s_{j+1},y}+\{s_{j-1}s_j,s_js_{j+1}\}$ and $P_{u,v}=P_{x,u}[u,r_{j-1}]+P_{r_{j},v}+\{r_{j-1}r_j\}$.

Otherwise, $j\geq m$. By Lemma \ref{hamilpath1} there is a spanning path $P_{x,u}$ in $Q[0,m-1]$ such that $E(P_{x,u})\cap E(Q[m-1])$ forms a spanning path or 2-path in $Q[m-1]$. Let $s_{m-1}r_{m-1}\in E(P_{x,u})\cap E(Q[m-1])$ where $s_{m-1}$ is closer to $x$ on $P_{x,u}$ than $r_{m-1}$. If $j=m$, we may choose $s_{m-1}r_{m-1}$ such that $p(s_{m-1})= p(y)$ and $\{v,y\}\cap\{s_m,r_m\}=\emptyset$. Then $p(s_m)\neq p(y)$ and $p(r_m)\neq p(v)$. By Lemma \ref{pathpartition3} there is a spanning 2-path $P_{s_m,y}+P_{r_m,v}$ in $Q[m,k-1]$. Let $P_{x,y}=P_{x,u}[x,s_{m-1}]+P_{s_m,y}+\{s_{m-1}s_m\}$ and $P_{u,v}=P_{x,u}[u,r_{m-1}]+P_{r_m,v}+\{r_{m-1}r_m\}$. If $j>m$, we may choose $s_{m-1}r_{m-1}$ such that $y\notin\{s_m,r_m\}$ and $p(s_{m-1})\neq p(y)$. Then $p(s_m)= p(y)\neq p(r_m)$ and $p(r_{m+1})\neq p(v)$. By Lemma \ref{delete} and Lemma \ref{hamilpath1} there exist spanning paths $P_{s_{m},y}$ in $Q[m]-r_m$ and $P_{r_{m+1},v}$ in $Q[m+1,k-1]$. Let $P_{x,y}=P_{x,u}[x,s_{m-1}]+P_{s_m,y}+\{s_{m-1}s_m\}$ and $P_{u,v}=P_{x,u}[u,r_{m-1}]+P_{r_{m+1},v}+\{r_{m-1}r_m,r_mr_{m+1}\}$.

In the above three cases, $P_{x,y}+P_{u,v}$ is the desired spanning 2-path in $Q_n^{k}$.
\end{proof}

\begin{Lemma}\label{pathpartition11} Let $n\geq5$ and $M$ be a matching in $Q_n^{k}$ with $|M|\leq 2n-10$. If $uu'$, $vv'$ and $ww'$ are three disjoint edges in $Q_n^{k}$ such that $\{u,v,w\}\cap V(M)=\emptyset$ and $\{u'v',u'w',v'w'\}\cap M=\emptyset$, then there is a spanning 3-path $P_{u,u'}+P_{v,v'}+P_{w,w'}$ passing through $M$ in $Q_n^{k}$.
\end{Lemma}

\begin{proof} We prove the lemma by induction on $n$. First, we prove the basis of the induction hypothesis. When $n=5$, now $M=\emptyset$. Split $Q_5^{k}$ into subcubes $Q[0],\ldots,Q[k-1]$ such that $xy\in E(Q[0])$, $uv\in E(Q[i])$ and $st\in E(Q[j])$. By symmetry we may assume $0\leq i\leq j\leq k-1$. Note that $x,y,u,v,s,t$ are distinct vertices and $p(x)\neq p(y)$, $p(u)\neq p(v)$, $p(s)\neq p(t)$ when $k$ is even.

If $0=i=j$, by Lemma \ref{pathpartition2} there exists a spanning 2-path $P_{x,y}'+P_{u,v}'$ passing through $st$ in $Q[0]$. Without loss of generality, assume $st\in E(P_{x,y}')$. Choose neighbors $s_{0}',t_{0}'$ of $s,t$ on $P_{x,y}'$ such that $s_{0}'\neq t$ and $t_{0}'\neq s$. Note that $s_{1}',t_{1}'$ are distinct and $p(s_{1}')\neq p(t_{1}')$ when $k$ is even. By Lemma \ref{hamilpath1} there exists a spanning path $P_{s_{1}',t_{1}'}$ in $Q[1,k-1]$. Let $P_{x,y}=P_{x,y}'+P_{s_{1}',t_{1}'}+\{s_{0}'s_{1}',t_{0}'t_{1}'\}-\{s,t\}$, $P_{u,v}=P_{u,v}'$ and $P_{s,t}=st$. If $0=i\neq j$ or $0\neq i=j$, without loss of generality, we may assume $0=i\neq j$. By Lemma \ref{pathpartition3} there exists a spanning 2-path $P_{x,y}+P_{u,v}$ in $Q[0,j-1]$ and by Lemma \ref{hamilpath1} there is a spanning path $P_{s,t}$ in $Q[j,k-1]$. If $i\neq0$, $j\neq0$ and $i\neq j$, without loss of generality, we may assume $0<i<j\leq k-1$. By Lemma \ref{hamilpath1} there exist spanning paths $P_{x,y}$ in $Q[0,i-1]$, $P_{u,v}$ in $Q[i,j-1]$ and $P_{s,t}$ in $Q[j,k-1]$, respectively. In the above three cases, $P_{x,y}+P_{u,v}+P_{s,t}$ is the desired spanning 3-path in $Q_5^{k}$.

Suppose that the lemma holds for $n-1$. We are to show that it holds for $n(\geq6)$. Select an integer $d\in\{1,\ldots,n\}$ such that $\{uu',vv',ww'\}\cap E_d=\emptyset$, and moreover, $|M\cap E_d|$ is as small as possible. Since $|M|\leq2n-10=2(n-3)-4$, we have $|M\cap E_d|\leq1$. When $|M\cap E_d|=1$, there are at least four ways to choose such an integer. Moreover, we may choose $d$ such that $\{u',v',w'\}\cap\{s_m,s_{m+1}\}=\emptyset$, where $\{s_ms_{m+1}\}=M\cap E_d$. Now $u,v,w,u',v',w',s_m,s_{m+1}$ are distinct vertices. Split $Q_n^{k}$ into subcubes $Q[0],\ldots,Q[k-1]$ by $E_d$. Since $uu'$, $vv'$ and $ww'$ are edges, by symmetry we may assume $uu'\in E(Q[0])$, $vv'\in E(Q[i])$ and $ww'\in E(Q[j])$, where $0\leq i\leq j\leq k-1$.

\textbf{Case 1.} $uu',vv'$ and $ww'$ lie in the same subcube. Without loss of generality assume all lie in $Q[0]$.

\textbf{Subcase 1.1.} $|M_0|\leq|M|-2$.

Since $|M_0|\leq|M|-2\leq2(n-1)-10$, by the induction hypothesis there is a spanning 3-path $P_{u,u'}^{0}+P_{v,v'}^{0}+P_{w,w'}^{0}$ passing through $M_0$ in $Q[0]$. Since $|M\cap E_d(0,1)|+|M\cap E_d(k-1,0)|\leq|M\cap E_d|\leq1$, by symmetry we may assume $M\cap E_d(k-1,0)=\emptyset$.

If $M\cap E_d(0,1)=\emptyset$, since $|E(P_{u,u'}^{0}+P_{v,v'}^{0}+P_{w,w'}^{0})|-|M_0|-|M_1|\geq k^{n-1}-3-(2n-10)>1$, there exists an edge $t_0t_0'\in E(P_{u,u'}^{0}+P_{v,v'}^{0}+P_{w,w'}^{0})\setminus M_0$ such that $t_1t_1'\notin M_1$. If $M\cap E_d(0,1)=\{s_0s_1\}$, let $t_0=s_0$ and $t_0t_0'$ be an edge on $P_{u,u'}^{0}+P_{v,v'}^{0}+P_{w,w'}^{0}$. Since $s_0s_1\in M$, we have $\{t_0t_0',t_1t_1'\}\cap M=\emptyset$.

In the two cases above, without loss of generality assume $t_0t_0'\in E(P_{u,u'}^{0})$. Since $|M_h|\leq2n-10<2n-4$ for every $h\geq1$ and
$|M\cap E_d|\leq1$, Lemma \ref{hamilpath2} implies that there is a spanning path $P_{t_1,t_1'}$ passing through $M[1,k-1]$ in $Q[1,k-1]$. Let $P_{u,u'}=P_{u,u'}^{0}+P_{t_1,t_1'}+\{t_0t_1,t_0't_1'\}-\{t_0t_0'\}$, $P_{v,v'}=P_{v,v'}^{0}$ and $P_{w,w'}=P_{w,w'}^{0}$. Hence, $P_{u,u'}+P_{v,v'}+P_{w,w'}$ is the desired spanning 3-path in $Q_n^{k}$.

\textbf{Subcase 1.2.} $|M|-1\leq|M_0|\leq|M|$. Now $|M\cap E_d|+\sum_{h=1}^{k-1}|M_h|\leq1$.

Since $M_0\cup\{uu',vv',ww'\}$ forms a linear forest with $|M_0\cup\{uu',vv',ww'\}|\leq2(n-1)-5$, Theorem \ref{forest} implies that there is a Hamiltonian cycle $C_0$ containing $M_0\cup\{uu',vv',ww'\}$ in $Q[0]$. Since $k\geq4$ and $|M\cap E_d|+\sum_{h=1}^{k-1}|M_h|\leq1$, by symmetry assume $M\cap E_d(k-1,0)=E_d(k-2,k-1)=M_{k-1}=\emptyset$.

The order of $u$ and $u'$(or $v$ and $v'$, or $w$ and $w'$) does not affect the proof. Without loss of generality, assume that the six points are arranged in the clockwise order of $u,u',v,v',w,w'$ in $C_0$; see Figure 1. Note that $C_0-\{uu',vv',ww'\}$ is a spanning 3-path passing through $M_0$ in $Q[0]$. Denote it by $P_{u,v'}+P_{v,w'}+P_{w,u'}$. Since $\{u,v,w\}\cap V(M)=\emptyset$ and $\{u'v',u'w',v'w'\}\cap M=\emptyset$, there exist edges $a_0a_0'\in E(P_{u,v'})\setminus M_0$, $b_0b_0'\in E(P_{v,w'})\setminus M_0$ and $c_0c_0'\in E(P_{w,u'})\setminus M_0$. Moreover, when $M\cap E_d(0,1)=\{s_0s_1\}$, we can choose the above three edges such that $s_0\in\{a_0,a_0',b_0,b_0',c_0,c_0'\}$. Without loss of generality, assume that $a_0$ is closer to $u$ than $a_0'$, and $b_0$ is closer to $v$ than $b_0'$, and $c_0$ is closer to $w$ than $c_0'$; see Figure 1. Now $P_{u,v'}+P_{v,w'}+P_{w,u'}-\{a_0a_0',b_0b_0',c_0c_0'\}$ is a spanning 6-path. Denote it by $P_{u,a_0}+P_{a_0',v'}+P_{v,b_0}+P_{b_0',w'}+P_{w,c_0}+P_{c_0',u'}$. Note that $a_0,a_0',b_0,b_0',c_0,c_0'$ are distinct. When $k$ is even, since $\{a_0,a_0',b_0,b_0',c_0,c_0'\}$ is balanced, at least one pair of vertices in $\{\{a_0,c_0'\},\{a_0'b_0\},\{b_0'c_0\}\}$, say $\{a_0,c_0'\}$, is balanced. So $\{a_0',b_0,b_0',c_0\}$ is also balanced.

First suppose $M\cap E_d(0,1)=\emptyset$. By Lemma \ref{pathpartition3} or Lemma \ref{pathpartition5} there is a spanning 2-path $P_{a_{k-1}',b_{k-1}}+P_{b_{k-1}',c_{k-1}}$ in $Q[k-1]$. If $a_1c_1'\notin M_1$, by Lemma \ref{hamilpathone} there is a spanning path $P_{a_{1},c_{1}'}$ passing through $M[1,k-2]$ in $Q[1,k-2]$. Let $P_{u,u'}=P_{u,a_0}+P_{c_0',u'}+P_{a_{1},c_{1}'}+\{a_0a_1,c_0'c_1'\}$, $P_{v,v'}=P_{a_0',v'}+P_{v,b_0}+P_{a_{k-1}',b_{k-1}}+\{a_0'a_{k-1}',b_0b_{k-1}\}$ and $P_{w,w'}=P_{w,c_0}+P_{b_0',w'}+P_{b_{k-1}',c_{k-1}}+\{c_0c_{k-1},b_0'b_{k-1}'\}$; see Figure 1. If $a_1c_1'\in M_1$, since $|E(P_{u,v'}+P_{v,w'}+P_{w,u'})|-|M_0\cup\{a_0a_0',b_0b_0',c_0c_0'\}|=k^{n-1}-3-(2n-10+3)>4$, we may choose an edge $r_0r_0'\in E(P_{u,v'}+P_{v,w'}+P_{w,u'})\setminus(M_0\cup\{a_0a_0',b_0b_0',c_0c_0'\})$ such that $\{r_1,r_1'\}\cap\{a_1,c_1'\}=\emptyset$. Without loss of generality assume $r_0r_0'\in E(P_{u,a_0})$. By Lemma \ref{pathpartition9} there is a spanning path $P_{r_1,r_1'}$ in $Q[1]-\{a_1,c_1'\}$. Choose an edge $t_1t_1'$ on $P_{r_1,r_1'}$, and by Lemma \ref{hamilpath2} there is a spanning path $P_{t_{2},t_{2}'}$ in $Q[2,k-2]$. Let $P_{u,u'}=P_{u,a_0}+P_{c_0',u'}+P_{r_1,r_1'}+P_{t_{2},t_{2}'}+\{a_0a_1,c_0'c_1',a_{1}c_{1}',r_0r_1,r_0'r_1',t_{1}t_{2},t_{1}'t_{2}'\}-\{r_0r_0',t_1t_1'\}$, $P_{v,v'}=P_{a_0',v'}+P_{v,b_0}+P_{a_{k-1}',b_{k-1}}+\{a_0'a_{k-1}',b_0b_{k-1}\}$ and $P_{w,w'}=P_{w,c_0}+P_{b_0',w'}+P_{b_{k-1}',c_{k-1}}+\{c_0c_{k-1},b_0'b_{k-1}'\}$. Now $P_{u,u'}+P_{v,v'}+P_{w,w'}$ is the desired spanning 3-path in $Q_n^{k}$.

\begin{figure}[h]
\begin{center}
\includegraphics[scale=0.5]{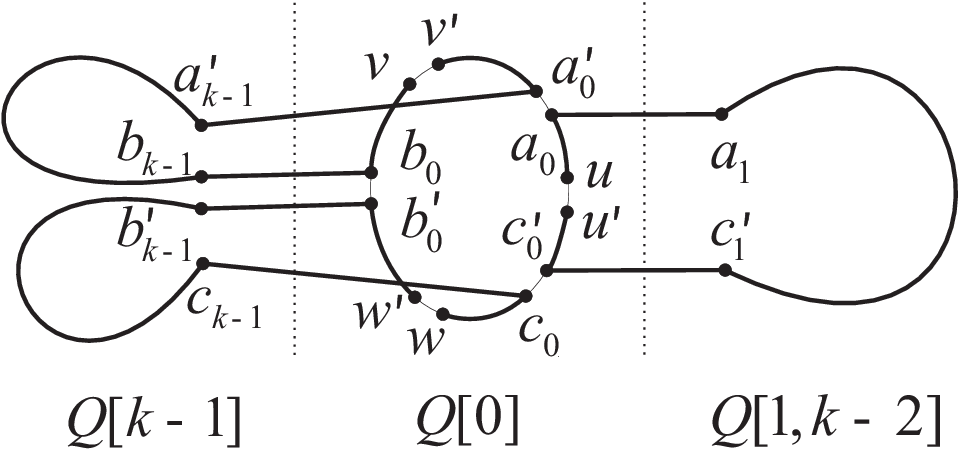}\\{Figure 1. Illustration for the proof of Subcase 1.2 in Lemma \ref{pathpartition11}.}
\end{center}
\end{figure}

Next suppose $M\cap E_d(0,1)=\{s_0s_1\}$. Now $M[1,k-2]=\emptyset$. If $s_0\in\{a_0,c_0'\}$, by Lemma \ref{hamilpath1} there is a spanning path $P_{a_{1},c_{1}'}$ in $Q[1,k-2]$, and by Lemma \ref{pathpartition3} or Lemma \ref{pathpartition5} there is a spanning 2-path $P_{a_{k-1}',b_{k-1}}+P_{b_{k-1}',c_{k-1}}$ in $Q[k-1]$. Let $P_{u,u'}=P_{u,a_0}+P_{c_0',u'}+P_{a_{1},c_{1}'}+\{a_0a_1,c_0'c_1'\}$, $P_{v,v'}=P_{a_0',v'}+P_{v,b_0}+P_{a_{k-1}',b_{k-1}}+\{a_0'a_{k-1}',b_0b_{k-1}\}$ and $P_{w,w'}=P_{w,c_0}+P_{b_0',w'}+P_{b_{k-1}',c_{k-1}}+\{c_0c_{k-1},b_0'b_{k-1}'\}$; see Figure 1. If $s_0\in\{a_0',b_0,b_0',c_0\}$, by Lemma \ref{pathpartition3} or Lemma \ref{pathpartition5} there is a spanning 2-path $P_{a_{1}',b_{1}}+P_{b_{1}',c_{1}}$ in $Q[1]$, and by Lemma \ref{hamilpath1} there is a spanning path $P_{a_{k-1},c_{k-1}'}$ in $Q[2,k-1]$. Let $P_{u,u'}=P_{u,a_0}+P_{c_0',u'}+P_{a_{k-1},c_{k-1}'}+\{a_0a_{k-1},c_0'c_{k-1}'\}$, $P_{v,v'}=P_{a_0',v'}+P_{v,b_0}+P_{a_{1}',b_{1}}+\{a_0'a_{1}',b_0b_{1}\}$ and $P_{w,w'}=P_{w,c_0}+P_{b_0',w'}+P_{b_{1}',c_{1}}+\{c_0c_{1},b_0'b_{1}'\}$. Now $P_{u,u'}+P_{v,v'}+P_{w,w'}$ is the desired spanning 3-path in $Q_n^{k}$.

\textbf{Case 2.} Two edges in $\{uu',vv',ww'\}$ lie in the same subcube. Without loss of generality assume $uu',vv'$ lie in $Q[0]$ and $ww'$ lie in $Q[j]$.

If $|M\cap E_d|=1$ and the edge in $M\cap E_d$ connects $Q[0]$ and $Q[j]$. Now $j=1$ and $M\cap E_d=\{s_0s_1\}$. Each edge in $M_0$ has at most one vertex adjacent to $s_0$; otherwise, it would form a 3-cycle, which contradicts the condition that $k\geq4$. Since $s_0$ has $2(n-1)$ neighbors in $V(Q[0])$ and $2(n-1)>2n-10+3$ , we may choose a neighbor $s_0'$ of $s_0$ in $V(Q[0])$ such that $s_0'\notin V(M_0)$, $s_1'\notin \{w,w'\}$ and $s_1'w'\notin M_1$. Then $M_0\cup\{s_0s_0'\}$ is a matching, $\{s_0s_0',s_1s_1'\}\cap M=\emptyset$ and $\{w,s_1\}\notin V(M_1)$. Since $|M_0\cup\{s_0s_0'\}|\leq2n-10+1\leq2(n-1)-7$ and $|M_1|\leq2n-10$, by Lemma \ref{pathpartition2} there exist spanning 2-paths $P_{u,u'}^{0}+P_{v,v'}^{0}$ passing through $M_0\cup\{s_0s_0'\}$ in $Q[0]$ and $P_{w,w'}^{1}+P_{s_1,s_1'}$ passing through $M_1$ in $Q[1]$, respectively. Since $|E(P_{w,w'}^{1}+P_{s_1,s_1'})|-|M_1|-|M_2|\geq k^{n-1}-2-(2n-10)>1$, we may choose an edge $r_1r_1'\in E(P_{w,w'}^{1}+P_{s_1,s_1'})\setminus M_1$ such that $r_2r_2'\notin M_2$. Since $|M_h|\leq2n-10<2n-4$ for every $h\in\{j,\ldots,k-1\}$, by Lemma \ref{hamilpath2} there is a spanning path $P_{r_2,r_2'}$ passing through $M[2,k-1]$ in $Q[2,k-1]$. Without loss of generality, assume $s_0s_0'\in E(P_{u,u'}^{0})$ and $r_1r_1'\in E(P_{w,w'}^{1})$. Let $P_{u,u'}=P_{u,u'}^{0}+P_{s_1,s_1'}+\{s_0s_1,s_0's_1'\}-\{s_0s_0'\}$, $P_{v,v'}=P_{v,v'}^{0}$ and $P_{w,w'}=P_{w,w'}^{1}+P_{r_2,r_2'}+\{r_1r_{2},r_1'r_{2}'\}-\{r_1r_1'\}$.

Otherwise, by symmetry we may assume $M\cap E_d(k-1,0)=M\cap E_d(j-1,j)=\emptyset$. By Lemma \ref{pathpartition2} there exists a spanning 2-path $P_{u,u'}^{0}+P_{v,v'}^{0}$ passing through $M_0$ in $Q[0]$ and by Lemma \ref{hamilpath2} there exists a spanning path $P_{w,w'}$ passing through $M[j,k-1]$ in $Q[j,k-1]$. If $j=1$, let $P_{u,u'}=P_{u,u'}^{0}$ and $P_{v,v'}=P_{v,v'}^{0}$. Otherwise, $j>1$. If $M\cap E_d(0,1)=\emptyset$, since $|E(P_{u,u'}^{0}+P_{v,v'}^{0})|-|M_0|-|M_1|\geq k^{n-1}-2-(2n-10)>1$, we may choose an edge $r_0r_0'\in E(P_{u,u'}^{0}+P_{v,v'}^{0})\setminus M_0$ such that $r_1r_1'\notin M_1$. Without loss of generality, assume $r_0r_0'\in E(P_{u,u'}^{0})$. If $M\cap E_d(0,1)=\{s_0s_1\}$, without loss of generality, assume $s_0\in V(P_{u,u'}^{0})$. Let $s_0=r_0$ and $r_0'$ be a neighbor of $r_0$ on $P_{u,u'}^{0}$. Since $s_0s_1\in M$, we have $\{r_0r_0',r_1r_1'\}\cap M=\emptyset$. By Lemma \ref{hamilpath2} there exists a spanning path $P_{r_1,r_1'}$ passing through $M[1,j-1]$ in $Q[1,j-1]$. Let $P_{u,u'}=P_{u,u'}^{0}+P_{r_1,r_1'}+\{r_0r_1,r_0'r_1'\}-\{r_0r_0'\}$ and $P_{v,v'}=P_{v,v'}^{0}$.

In the above two cases, $P_{u,u'}+P_{v,v'}+P_{w,w'}$ is the desired spanning 3-path passing through $M$ in $Q_n^{k}$.

\textbf{Case 3.} $uu',vv'$ and $ww'$ lie in different subcubes. Without loss of generality, we assume $0<i<j\leq k-1$.

If $|M\cap E_d|=1$ and the edge in $M\cap E_d$ connects two of $Q[0]$, $Q[i]$ and $Q[j]$. Without loss of generality, assume it connects $Q[0]$ and $Q[i]$. Now $i=1$ and $M\cap E_d=\{s_0s_1\}$. Since $s_0$ has $2(n-1)$ neighbors in $V(Q[0])$ and $2(n-1)>2n-10+3$ , we may choose a neighbor $s_0'$ of $s_0$ in $V(Q[0])$ such that $s_0'\notin V(M_0)$, $s_1'\notin \{v,v'\}$ and $s_1'v'\notin M_1$. Then $M_0\cup\{s_0s_0'\}$ is a matching, $\{s_0s_0',s_1s_1'\}\cap M=\emptyset$ and $\{v,s_1\}\notin V(M_1)$. By Lemma \ref{hamilpath2} there exist spanning paths $P_{u,u'}^{0}$ passing through $M_0\cup\{s_0s_0'\}$ in $Q[0]$ and $P_{w,w'}$ passing through $M[j,k-1]$ in $Q[j,k-1]$. By Lemma \ref{pathpartition2} there exists a spanning 2-path $P_{v,v'}^{1}+P_{s_1,s_1'}$ passing through $M_1$ in $Q[1]$. If $j=2$, let $P_{u,u'}=P_{u,u'}^{0}+P_{s_1,s_1'}+\{s_0s_1,s_0's_1'\}-\{s_0s_0'\}$ and $P_{v,v'}=P_{v,v'}^{1}$. If $j>2$, since $|E(P_{v,v'}^{1}+P_{s_1,s_1'})|-|M_1|-|M_2|\geq k^{n-1}-2-(2n-10)>1$, we may choose an edge $r_1r_1'\in E(P_{v,v'}^{1}+P_{s_1,s_1'})\setminus M_1$ such that $r_2r_2'\notin M_2$. By Lemma \ref{hamilpath2} there is a spanning path $P_{r_2,r_2'}$ passing through $M[2,j-1]$ in $Q[2,j-1]$. Without loss of generality, assume $r_1r_1'\in E(P_{v,v'}^{1})$. Let $P_{u,u'}=P_{u,u'}^{0}+P_{s_1,s_1'}+\{s_0s_1,s_0's_1'\}-\{s_0s_0'\}$ and $P_{v,v'}=P_{v,v'}^{1}+P_{r_2,r_2'}+\{r_1r_{2},r_1'r_{2}'\}-\{r_1r_1'\}$. Now $P_{u,u'}+P_{v,v'}+P_{w,w'}$ is a spanning 3-path passing through $M$ in $Q_n^{k}$.

Otherwise, $|M\cap E_d|=1$ and the edge in $M\cap E_d$ connects at most one of $Q[0]$, $Q[i]$ and $Q[j]$, or $|M\cap E_d|=0$. Now by symmetry, we may assume $M\cap E_d(k-1,0)=M\cap E_d(i-1,i)=M\cap E_d(j-1,j)=\emptyset$. By Lemma \ref{hamilpath2} there exist spanning paths $P_{u,u'}$ passing through $M[0,i-1]$ in $Q[0,i-1]$, $P_{v,v'}$ passing through $M[i,j-1]$ in $Q[i,j-1]$ and $P_{w,w'}$ passing through $M[j,k-1]$ in $Q[j,k-1]$, respectively. Hence, $P_{u,u'}+P_{v,v'}+P_{w,w'}$ is a spanning 3-path passing through $M$ in $Q_n^{k}$.
\end{proof}

\begin{Lemma}\label{pathpartition12} For $n\geq4$, let $x,y$ be distinct vertices in $Q_n^{k}$ with $d_{Q_n^{k}}(x,y)\leq3$, and $p(x)\neq p(y)$ whenever $k$ is even. If $M$ is a matching in $Q_n^{k}$ such that $|M|\leq2n-8$ and $xy\notin M$, then there is a spanning $x,y$-path passing through $M$ in $Q_n^{k}$.
\end{Lemma}

\begin{proof} We prove the Lemma by induction on $n$. The Lemma holds for $n\leq5$ by Lemma \ref{pathpartition8}. Now consider $n\geq6$ and suppose the Lemma holds for $n-1$. Since $d_{Q_n^{k}}(x,y)\leq3$, $x,y$ are different at most in three positions. Since $|M|\leq2(n-3)-2$, there exists an integer $d\in\{1,\ldots,n\}$ such that $|M\cap E_d|\leq1$ and $x,y$ lie in the same subcube. Split $Q_n^{k}$ into subcubes $Q[0],\ldots,Q[k-1]$ by $E_d$. By symmetry we may assume $x,y\in V(Q[0])$ and $M\cap E_d(k-1,0)=\emptyset$.

\textbf{Case 1.} $|M_0|\leq2n-10$.

Since $|M_0|\leq2(n-1)-8$, by the induction hypothesis there is a spanning path $P_{x,y}'$ passing through $M_0$ in $Q[0]$. If $M\cap E_d(0,1)=\emptyset$, since $|E(P_{x,y}')|-|M_0|-|M_{1}|\geq k^{n-1}-1-(2n-8)>1$, there exists an edge $w_0w_0'\in E(P_{x,y}')\setminus M_0$ such that $w_{1}w_{1}'\notin M_{1}$. If $M\cap E_d(0,1)\neq\emptyset$, let $M\cap E_d(0,1)=\{w_0w_{1}\}$ and $w_0'$ be a neighbor of $w_0$ on $P_{x,y}'$, then $\{w_0w_0',w_{1}w_{1}'\}\cap M=\emptyset$. Since $|M_i|\leq2n-8<2n-4$ for $1\leq i\leq k-1$ and $|M\cap E_d|\leq1$, by Lemma \ref{hamilpath2} there is a spanning path $P_{w_{1},w_{1}'}$ passing through $M[1,k-1]$ in $Q[1,k-1]$. Hence, $P_{x,y}=P_{x,y}'+P_{w_{1},w_{1}'}+\{w_0w_{1},w_0'w_{1}'\}-\{w_0w_0'\}$ is a spanning path passing through $M$ in $Q_n^{k}$.

\textbf{Case 2.} $|M_0|=2n-9$. Now $|M\cap E_d(0,1)|+|M[1,k-1]|\leq1$.

Let $u_0v_0\in M_0$ such that $\{u_0,v_0\}\cap\{x,y\}=\emptyset$. By the induction hypothesis there is a spanning path $P_{x,y}'$ passing through $M_0\setminus\{u_0v_0\}$ in $Q[0]$. When $u_0v_0\in E(P_{x,y}')$, the conclusion holds by Case 1. When $u_0v_0\notin E(P_{x,y}')$, choose neighbors $u_0',v_0'$ of $u_0,v_0$ on $P_{x,y}'$, respectively, such that exactly one of $\{u_0,v_0\}$ lies on $P_{x,y}'[u_0,v_0]$. Since $\{u_0,v_0\}\cap \{x,y\}=\emptyset$, there are two choices of the above $u_0',v_0'$. Since $M$ is a matching and $u_0v_0\in M_0$, we have $\{u_0u_0',v_0v_0'\}\cap M=\emptyset$. Note that $u_0',v_0'$ are distinct, and whenever $k$ is even, $p(u_0')\neq p(v_0')$.

If $M\cap E_d(0,1)=\emptyset$, then since $|M[1,k-1]|\leq1$, we may choose $u_0'$ and $v_0'$ such that $u_1'v_1'\notin M_1$. If $M\cap E_d(0,1)=\{w_0w_{1}\}$ and $d_{P_{x,y}'}(w_0,u_0v_0)=1$, then we may choose $u_0',v_0'$ such that $w_0\in\{u_0',v_0'\}$. Now $|M[1,k-1]|=0$. Note that $u_1'$ and $v_1'$ are distinct and $p(u_1')\neq p(v_1')$ when $k$ is even. By Lemma \ref{hamilpathone} there exists a spanning path $P_{u_1',v_1'}$ passing through $M[1,k-1]$ in $Q[1,k-1]$. Hence, $P_{x,y}=P_{x,y}'+P_{u_1',v_1'}+\{u_0v_0,u_0'u_{1}',v_0'v_{1}'\}-\{u_0u_0',v_0v_0'\}$ is the desired spanning path in $Q_n^{k}$. If $M\cap E_d(0,1)=\{w_0w_{1}\}$ and $d_{P_{x,y}'}(w_0,u_0v_0)>1$, then we may choose neighbors $u_0',v_0',w_0'$ of $u_0,v_0,w_0$ on $P_{x,y}'$, respectively, such that $u_0',v_0',w_0,w_0'$ are distinct. Recall that $|M[1,k-1]|=0$. By Lemma \ref{pathpartition3} there is a spanning 2-path $P_{u_1',v_1'}+P_{w_1,w_1'}$ in $Q[1,k-1]$. Now $P_{x,y}=P_{x,y}'+P_{u_1',v_1'}+P_{w_1,w_1'}+\{u_0v_0,u_0'u_{1}',v_0'v_{1}',w_0w_1,w_0'w_1'\}-\{u_0u_0',v_0v_0',w_0w_0'\}$ is the desired spanning path in $Q_n^{k}$; see Figure 2.

\begin{figure}[h]
\begin{center}
\includegraphics[scale=0.5]{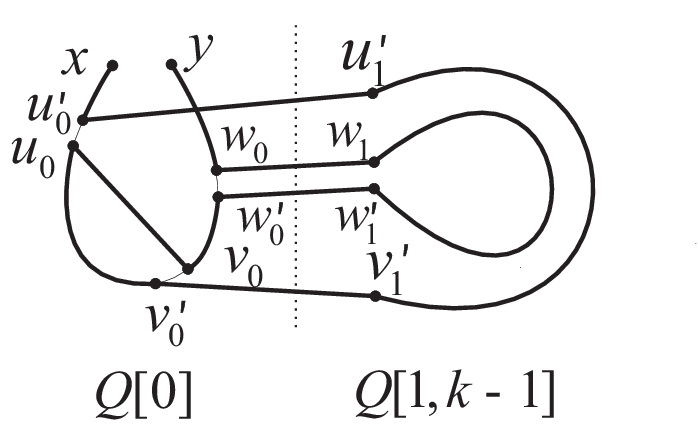}\\{Figure 2. Illustration for the proof of Case 2 in Lemma \ref{pathpartition12} when $M\cap E_d(0,1)=\{w_0w_{1}\}$ and $d_{P_{x,y}'}(w_0,u_0v_0)>1$.}
\end{center}
\end{figure}

\begin{figure}[h]
\begin{center}
\includegraphics[scale=0.5]{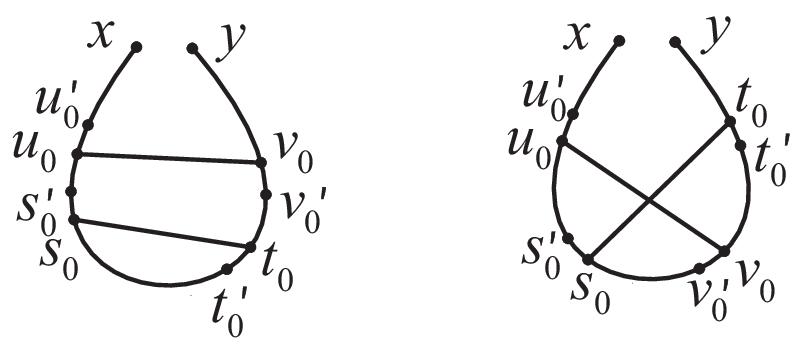}\\{Figure 3. The two possibilities for $\{u_0v_0,s_0t_0\}$ on $P_{x,y}'$ of Case 3 in Lemma \ref{pathpartition12}.}
\end{center}
\end{figure}

\textbf{Case 3.} $|M_0|=2n-8$. Now $M\cap E_d(0,1)=M[1,k-1]|=\emptyset$.

Let $\{u_0v_0,s_0t_0\}\subseteq M_0$ such that $\{u_0,v_0,s_0,t_0\}\cap\{x,y\}=\emptyset$. By the induction hypothesis there is a spanning path $P_{x,y}'$ passing through $M_0\setminus\{u_0v_0,s_0t_0\}$ in $Q[0]$. When $\{u_0v_0,s_0t_0\}\cap E(P_{x,y}')\neq\emptyset$, the conclusion holds by Case 1 or Case 2. When $\{u_0v_0,s_0t_0\}\cap E(P_{x,y}')=\emptyset$, there are two possibilities up to isomorphism for $\{u_0v_0,s_0t_0\}$ on $P_{x,y}'$; see Figure 3. Choose neighbors $u_0',v_0',s_0',t_0'$ of $u_0,v_0,s_0,t_0$ on $P_{x,y}'$, respectively, which are closer to $x$. Since $M$ is a matching and $\{u_0v_0,s_0t_0\}\subseteq M$, we have $\{u_0u_0',v_0v_0',s_0s_0',t_0t_0'\}\cap M=\emptyset$. Note that $u_0',v_0',s_0',t_0'$ are distinct and whenever $k$ is even, $p(u_0')\neq p(v_0')$ and $p(s_0')\neq p(t_0')$. Then $P_{x,y}'-\{u_0u_0',v_0v_0',s_0s_0',t_0t_0'\}+\{u_0v_0,s_0t_0\}$ is a spanning 3-path passing through $M$ in $Q[0]$, denoted by $P_{x,a_0}+P_{y,b_0}+P_{c_0,d_0}$. Now $\{a_0,b_0,c_0,d_0\}=\{u_0',v_0',s_0',t_0'\}$, which is balanced when $k$ is even. By Lemma \ref{pathpartition3} or Lemma \ref{pathpartition5} there is a spanning 2-path $P_{a_1,c_1}+P_{b_1,d_1}$ in $Q[1]$. Choose an edge $r_1r_1'\in E(P_{a_1,c_1}+P_{b_1,d_1})$. By Lemma \ref{hamilpath1}  there is a spanning path $P_{r_2,r_2'}$ in $Q[2,k-1]$. Now $P_{x,y}=P_{x,a_0}+P_{y,b_0}+P_{c_0,d_0}+P_{a_1,c_1}+P_{b_1,d_1}+P_{r_2,r_2'}+\{a_0a_1,b_0b_1,c_0c_1,d_0d_1,r_1r_2,r_1'r_2'\}-\{r_1r_1'\}$ is the desired spanning path in $Q_n^{k}$.
\end{proof}

\section{Proof of Theorem \ref{mosttheorem}}

We prove the theorem by induction on $n$. The theorem holds for $n\leq12$ by Theorem \ref{maintheorem}. Suppose that the theorem holds for $n-1(\geq12)$, we are to show that it holds for $n(\geq13)$.

Since $|M|\leq4n-20$, there exists an integer $d\in\{1,\ldots,n\}$ such that $|M\cap E_d|\leq3$. Split $Q_n^{k}$ into subcubes $Q[0],\ldots,Q[k-1]$ by $E_d$. By symmetry we may assume $|M_0|\geq|M_i|$ for every $i\geq1$. Now $|M_i|\leq\lfloor\frac{4n-20}{2}\rfloor=2n-10<2n-4$ for every $i\geq1$.

\vskip 0.2cm

\textbf{Claim 1.} \emph{If there is a Hamiltonian cycle $C_0$ containing $M_0$ in $Q[0]$, then we can construct a Hamiltonian cycle containing $M$ in $Q_n^{k}$.}

\vskip 0.2cm

\textbf{Case 1.} $\forall i\in\{0,\ldots,k-1\}$, $|M\cap E_d(i,i+1)|\leq1$.

If $M\cap E_d(0,1)=\emptyset$ or $M\cap E_d(k-1,0)=\emptyset$, by symmetry we may assume $M\cap E_d(k-1,0)=\emptyset$. When $M\cap E_d(0,1)=\emptyset$, since $|E(C_0)|-|M_0|-|M_1|\geq k^{n-1}-(4n-20)>1$, there exists an edge $w_0w_0'\in E(C_0)\setminus M_0$ such that $w_1w_1'\notin M_1$. When $M\cap E_d(0,1)\neq\emptyset$, let $M\cap E_d(0,1)=\{w_0w_1\}$ and $w_0'$ be a neighbor of $w_0$ on $C_0$, thus $w_0w_0'\notin M_0$ and $w_1w_1'\notin M_1$. By Lemma \ref{hamilpath2} there exists a spanning path $P_{w_1,w_1'}$ passing through $M[1,k-1]$ in $Q[1,k-1]$. Hence, $C_0+P_{w_1,w_1'}+\{w_0w_1,w_0'w_1'\}-\{w_0w_0'\}$ is the desired Hamiltonian cycle in $Q_n^{k}$.

Otherwise, $|M\cap E_d(0,1)|=1$ and $|M\cap E_d(k-1,0)|=1$. Let $M\cap E_d(0,1)=\{a_{0}a_{1}\}$ and $M\cap E_d(k-1,0)=\{c_{k-1}c_{0}\}$. We choose neighbor $a_{0}'$ and $c_{0}'$ of $a_{0}$ and $c_{0}$ on $C_0$ such that $a_{0}a_{0}'\neq c_{0}c_{0}'$. Then $\{a_{0}a_{0}',c_{0}c_{0}',a_{1}a_{1}',c_{k-1}c_{k-1}'\}\cap M=\emptyset$. Since $|M\cap E_d|\leq3$ and $k\geq4$, there exists an integer $j\in\{1,\ldots,k-2\}$ such that $M\cap E_d(j,j+1)=\emptyset$. Lemma \ref{hamilpath2} implies that there exist spanning paths $P_{a_1,a_1'}$ passing through $M[1,j]$ in $Q[1,j]$ and $P_{c_{k-1},c_{k-1}'}$ passing through $M[j+1,k-1]$ in $Q[j+1,k-1]$, respectively. Now $C_0+P_{a_1,a_1'}+P_{c_{k-1},c_{k-1}'}+\{a_0a_1,a_0'a_1',c_{0}c_{k-1},c_{0}'c_{k-1}'\}-\{a_{0}a_{0}',c_{0}c_{0}'\}$ is the desired Hamiltonian cycle in $Q_n^{k}$.

\textbf{Case 2.} $\exists j\in\{0,\ldots,k-1\}$, $|M\cap E_d(j,j+1)|=2$.

\textbf{Subcase 2.1.} $j=0$ or $j=k-1$. By symmetry we may assume $j=0$.

Let $M\cap E_d(0,1)=\{a_{0}a_{1},b_{0}b_{1}\}$. First, we find a Hamiltonian cycle $C[0,1]$ containing $M[0,1]$ in $Q[0,1]$. If $d_{C_0}(a_0,b_0)=1$, since $\{a_{0}a_{1},b_{0}b_{1}\}\subseteq M$, we have $\{a_{0}b_{0},a_{1}b_{1}\}\cap M=\emptyset$. Since $|M_{1}|<2n-4$, Lemma \ref{hamilpath2} implies that there is a spanning path $P_{a_{1},b_{1}}$ passing through $M_{1}$ in $Q[1]$. Let $C[0,1]=C_0+P_{a_{1},b_{1}}+\{a_{0}a_{1},b_{0}b_{1}\}-\{a_{0}b_{0}\}$. If $d_{C_0}(a_0,b_0)>1$, then denote the two neighbors of $a_{0}$ on $C_0$ by $a_0'$ and $a_0''$. We choose a neighbor $b_0'$ of $b_0$ on $C_0$ such that $b_0'\notin\{a_0',a_0''\}$. Since $a_{1}'b_{1}'$ and $a_{1}''b_{1}'$ cannot be in $M_{1}$ simultaneously, without loss of generality, assume $a_{1}'b_{1}'\notin M_{1}$. Since $|M_{1}|<2(n-1)-7$, by Lemma \ref{pathpartition2} there is a spanning 2-path $P_{a_{1},a_{1}'}+P_{b_{1},b_{1}'}$ passing through $M_{1}$ in $Q[1]$. Let $C[0,1]=C_0+P_{a_{1},a_{1}'}+P_{b_{1},b_{1}'}+\{a_{0}a_{1},a_{0}'a_{1}',b_{0}b_{1},b_{0}'b_{1}'\}-\{a_{0}a_0',b_{0}b_0'\}$.

Since $|M\cap E_d(1,2)|+|M\cap E_d(k-1,0)|\leq1$, by symmetry, assume $M\cap E_d(k-1,0)=\emptyset$. If $M\cap E_d(1,2)=\emptyset$, since $|E(C[0,1])\cap E(Q[1])|-|M_{1}|-|M_{2}|\geq k^{n-1}-2-(4n-22)>1$, we may choose an edge $s_{1}s_{1}'\in E(C[0,1])\cap E(Q[1])\setminus M_{1}$ such that $s_{2}s_{2}'\notin M_{2}$. If $|M\cap E_d(1,2)|=1$, let $M\cap E_d(1,2)=\{s_{1}s_{2}\}$ and $s_{1}'$ be a neighbor of $s_{1}$ on $E(C[0,1])\cap E(Q[1])$, then $\{s_{1}s_{1}',s_{2}s_{2}'\}\cap M=\emptyset$. Since $|M_{i}|<2n-4$ for every $i\in\{2,\ldots,k-1\}$, Lemma \ref{hamilpath2} implies that there is a spanning path $P_{s_{2},s_{2}'}$ passing through $M[2,k-1]$ in $Q[2,k-1]$. Hence, $C[0,1]+P_{s_{2},s_{2}'}+\{s_{1}s_{2},s_{1}'s_{2}'\}-\{s_{1}s_{1}'\}$ is the desired Hamiltonian cycle in $Q_n^{k}$.

\textbf{Subcase 2.2.} $1\leq j\leq k-2$.

Let $M\cap E_d(j,j+1)=\{a_{j}a_{j+1},b_{j}b_{j+1}\}$. Since $\sum_{i=0}^{j-1}|M\cap E_d(i,i+1)|+\sum_{i=j+1}^{k-1}|M\cap E_d(i,i+1)|\leq1$, by symmetry we can assume $\sum_{i=j+1}^{k-1}|M\cap E_d(i,i+1)|=0$.

First, we find a Hamiltonian cycle $C[0,j-1]$ containing $M[0,j-1]$ in $Q[0,j-1]$. If $j=1$, let $C[0,j-1]=C_0$. Otherwise, $1<j\leq k-2$. When $M\cap E_d(0,1)=\emptyset$, since $|E(C_0)|-|M_{0}|-|M_{1}|\geq k^{n-1}-(4n-20-2)>1$, we may choose an edge $r_{0}r_{0}'\in E(C_0)\setminus M_{0}$ such that $r_{1}r_{1}'\notin M_{1}$. When $|M\cap E_d(0,1)|=1$, let $M\cap E_d(0,1)=\{r_0r_1\}$ and $r_{0}'$ be a neighbor of $r_{0}$ on $C_0$, thus $\{r_0r_0',r_1r_1'\}\cap M=\emptyset$. Note that $|M_i|<2n-4$ for every $i\geq1$. Lemma \ref{hamilpath2} implies that there is a spanning path $P_{r_1,r_1'}$ passing through $M[1,j-1]$ in $Q[1,j-1]$ such that $E(P_{r_1,r_1'})\cap E(Q[j-1])$ forms a spanning path in $Q[j-1]$. Let $C[0,j-1]=C_0+P_{r_1,r_1'}+\{r_0r_1,r_0'r_1'\}-\{r_0r_0'\}$.

Next, we construct a Hamiltonian cycle $C[0,j+1]$ containing $M[0,j+1]$ in $Q[0,j+1]$. If $M\cap E_d(j-1,j)=\emptyset$, since $|E(C[0,j-1])\cap E(Q[j-1])\setminus M_{j-1}|-2\cdot|V(M_{j})\cup\{a_{j},b_{j}\}|\geq k^{n-1}-1-(4n-20)-2(2(2n-10)+2)>k^{n-1}-1-2(2(4n-20))>1$, there exists an edge $s_{j-1}s_{j-1}'\in E(C[0,j-1])\cap E(Q[j-1])\setminus M_{j-1}$ such that $\{s_{j},s_{j}'\}\cap(V(M_{j})\cup\{a_{j},b_{j}\})=\emptyset$. If $|M\cap E_d(j-1,j)|=1$, let $M\cap E_d(j-1,j)=\{s_{j-1}s_{j}\}$ and $s_{j-1}'$ be a neighbor of $s_{j-1}$ on $E(C[0,j-1])\cap E(Q[j-1])$, thus $\{s_{j-1}s_{j-1}',s_{j}s_{j}'\}\cap M=\emptyset$ and $s_{j}\notin(V(M_{j})\cup\{a_{j},b_{j}\})$. Note that $s_{j}'$ could be $a_{j}$ or $b_{j}$, without affecting the following proof. Since $Q_n^{k}$ contains no 3-cycle when $k\geq4$, each edge in $Q_n^{k}$ has at most one vertex adjacent to $a_{j}$, and the same holds true for $b_{j}$. Since $a_{j}$ has $2(n-1)>2n-10+2$ neighbors in $Q[j]$, we may choose one, denoted by $a_{j}'$, such that $a_{j}'\notin V(M_{j})\cup\{b_{j},s_{j},s_{j}'\}$. Since $b_{j}$ has $2(n-1)>2n-10+3$ neighbors in $Q[j]$, we may choose one, denoted by $b_{j}'$, such that $b_{j}'\notin V(M_{j})\cup\{a_{j},a_{j}',s_{j},s_{j}'\}$ and $a_{j+1}'b_{j+1}'\notin M_{j+1}$. Then $M_{j}\cup\{s_{j}s_{j}',a_{j}a_{j}',b_{j}b_{j}'\}$ is a linear forest in $Q[j]$ with $|M_{j}\cup\{s_{j}s_{j}',a_{j}a_{j}',b_{j}b_{j}'\}|\leq2n-10+3<2(n-1)-1$. By Theorem \ref{forest} there is a Hamiltonian cycle $C_j$ containing $M_{j}\cup\{s_{j}s_{j}',a_{j}a_{j}',b_{j}b_{j}'\}$ in $Q[j]$. Since $|M_{j+1}|<2(n-1)-7$, by Lemma \ref{pathpartition2} there is a spanning 2-path $P_{a_{j+1},a_{j+1}'}+P_{b_{j+1},b_{j+1}'}$ passing through $M_{j+1}$ in $Q[j+1]$. Let $C[0,j+1]=C[0,j-1]+C_j+P_{a_{j+1},a_{j+1}'}+P_{b_{j+1},b_{j+1}'}+\{s_{j-1}s_{j},s_{j-1}'s_{j}',a_{j}a_{j+1},a_{j}'a_{j+1}',b_{j}b_{j+1},$
$b_{j}'b_{j+1}'\}-\{s_{j-1}s_{j-1}',s_{j}s_{j}',a_{j}a_{j}',b_{j}b_{j}'\}$.

If $j=k-2$, then $C[0,j+1]$ is the desired Hamiltonian cycle in $Q_n^{k}$. If $j<k-2$, since $|E(C[0,j+1])\cap E(Q[j+1])|-|M_{j+1}|-|M_{j+2}|\geq k^{n-1}-2-(4n-22)>1$, there exists an edge $t_{j+1}t_{j+1}'\in E(C[0,j+1])\cap E(Q[j+1])\setminus M_{j+1}$ such that $t_{j+2}t_{j+2}'\notin M_{j+2}$. Lemma \ref{hamilpath2} implies that there is a spanning path $P_{t_{j+2},t_{j+2}'}$ passing through $M[j+2,k-1]$ in $Q[j+2,k-1]$. Now $C[0,j+1]+P_{t_{j+2},t_{j+2}'}+\{t_{j+1}t_{j+2},t_{j+1}'t_{j+2}'\}-\{t_{j+1}t_{j+1}'\}$ is the desired Hamiltonian cycle in $Q_n^{k}$.

\textbf{Case 3.} $\exists j\in\{0,\ldots,k-1\}$, $|M\cap E_d(j,j+1)|=3$.

Let $M\cap E_d(j,j+1)=\{a_{j}a_{j+1},b_{j}b_{j+1},c_{j}c_{j+1}\}$. Then $M\cap E_d(i,i+1)=\emptyset$ for every $ i\in\{0,\ldots,k-1\}\setminus\{j\}$.

\textbf{Subcase 3.1.} $j=0$ or $j=k-1$.

By symmetry we may assume $j=0$. Now $M\cap E_d(0,1)=\{a_{0}a_{1},b_{0}b_{1},c_{0}c_{1}\}$.

\textbf{Subcase 3.1.1.} $d_{C_0}(a_0,b_0)=1$, or $d_{C_0}(a_0,c_0)=1$, or $d_{C_0}(b_0,c_0)=1$.

By symmetry we assume $d_{C_0}(a_0,b_0)=1$. Choose a neighbor $c_{0}'$ of $c_{0}$ on $C_0$ such $c_{0}'\notin\{a_0,b_0\}$. Note that $|M_{1}|<2(n-1)-7$ and $\{a_{1},b_{1},c_{1}\}\cap V(M_{1})=\emptyset$. By Lemma \ref{pathpartition2} there is a spanning 2-path $P_{a_{1},b_{1}}+P_{c_{1},c_{1}'}$ passing through $M_{1}$ in $Q[1]$. Since $|E(C_0)|-|M_{0}|-|M_{k-1}|\geq k^{n-1}-1-(4n-23)>3$, we may choose an edge $s_{0}s_{0}'\in E(C_0)\setminus (M_{0}\cup\{a_{0}b_{0},c_{0}c_{0}'\})$ such that $s_{k-1}s_{k-1}'\notin M_{k-1}$. Lemma \ref{hamilpath2} implies that there is a spanning path $P_{s_{k-1}s_{k-1}'}$ passing through $M[2,k-1]$ in $Q[2,k-1]$. Hence, $C_0+P_{a_{1},b_{1}}+P_{c_{1},c_{1}'}+P_{s_{k-1}s_{k-1}'}+\{a_{0}a_{1},b_{0}b_{1},c_{0}c_{1},c_{0}'c_{1}',s_{0}s_{k-1},s_{0}'s_{k-1}'\}-$
$\{a_0b_0,c_{0}c_{0}',s_{0}s_{0}'\}$ is the desired Hamiltonian cycle in $Q_n^{k}$.

\textbf{Subcase 3.1.2.} $d_{C_0}(a_0,b_0)>1$, and $d_{C_0}(a_0,c_0)>1$, and $d_{C_0}(b_0,c_0)>1$.

Denote the two neighbors of $a_{0}$ on $C_0$ by $a_0'$ and $a_0''$. We choose a neighbor $b_0'$ of $b_0$ on $C_0$ such that $b_0'\notin\{a_0',a_0''\}$. Since $a_{1}'b_{1}'$ and $a_{1}''b_{1}'$ cannot be in $M_{1}$ simultaneously, without loss of generality, assume $a_{1}'b_{1}'\notin M_{1}$. Note that $a_{0},a_0',b_0,b_0',c_0$ are distinct and $|M_{1}|\leq2(n-1)-7$. By Lemma \ref{pathpartition2} there is a spanning 2-path $P_{a_{1},a_{1}'}+P_{b_{1},b_{1}'}$ passing through $M_{1}$ in $Q[1]$. Thus $C[0,1]=C_0+P_{a_{1},a_{1}'}+P_{b_{1},b_{1}'}+\{a_0a_1,a_0'a_{1}',b_0b_{1},b_0'b_1'\}-\{a_0a_0',b_0b_0'\}$ is a Hamiltonian cycle containing $M[0,1]\setminus\{c_0c_{1}\}$ in $Q[0,1]$; see Figure 4. Recall that $|M_{i}|\leq2n-10$ for every $2\leq i\leq k-1$, and $M\cap E_d(i,i+1)=\emptyset$ for every $ i\in\{1,\ldots,k-1\}$.

\textbf{Claim A.} \emph{Let $2\leq p<q\leq k-1$, and whenever $k$ is even, $q-p$ is odd. In $Q[p,q]$, if $x_p$ and $y_q$ are neighbors of $c_p$ and $c_q$ in $Q[p]$ and $Q[q]$, respectively, then there is a spanning path $P_{x_p,y_q}$ passing through $M[p,q]$ in $Q[p,q]$.}

We prove the claim by induction on $q-p$. When $q-p=1$, choose a neighbor $s_p$ of $x_p$ in $Q[p]$ such that $x_ps_p\notin M_p$, and $s_q\neq y_q$, and $s_qy_q\notin M_q$. Since $x_p$ has $2(n-1)>3$ neighbors in $Q[p]$, the above vertex $s_p$ exists. Now $d_{Q[p]}(x_p,s_p)=1$ and $d_{Q[q]}(s_q,y_q)\leq3$, and whenever $k$ is even, $p(x_p)=p(s_q)\neq p(s_p)=p(y_q)$. By Lemma \ref{pathpartition12}, there exist spanning paths $P_{x_p,s_p}$ passing through $M_{p}$ in $Q[p]$ and $P_{s_q,y_q}$ passing through $M_q$ in $Q[q]$. Let $P_{x_p,y_q}=P_{x_p,s_p}+s_ps_q+P_{s_q,y_q}$.

Suppose the claim holds for all $q'-p'<q-p$, where whenever $k$ is even, $q'-p'$ is odd. We are to show that it holds for $q-p>1$. When $k$ is even, since $q-p$ is odd, we have $q-(p+2)$ is odd. By induction, there exist spanning paths $P_{x_p,x_{p+1}}$ passing through $M[p,p+1]$ in $Q[p,p+1]$ and $P_{x_{p+2},y_q}$ passing through $M[p+2,q]$ in $Q[p+2,q]$. Let $P_{x_p,y_q}=P_{x_p,x_{p+1}}+x_{p+1}x_{p+2}+P_{x_{p+2},y_q}$. When $k$ is odd, choose a neighbor $s_p$ of $c_p$ in $Q[p]$ such that $s_p\neq x_p$. Now $d_{Q[p]}(x_p,s_p)=2$. By Lemma \ref{pathpartition12} and by induction, there exist spanning paths $P_{x_p,s_p}$ passing through $M[p]$ in $Q[p]$ and $P_{s_{p+1},y_q}$ passing through $M[p+1,q]$ in $Q[p+1,q]$. Let $P_{x_p,y_q}=P_{x_p,s_p}+s_ps_{p+1}+P_{s_{p+1},y_q}$.

Claim A is proved.

Choose neighbors $c_0'$ of $c_0$ and $c_1''$ of $c_1$ on $C[0,1]$, respectively, such that one of the paths joining $c_0$ and $c_1$ on $C[0,1]$ contains $c_0'$ and the other contains $c_1''$. Since $c_0\notin\{a_{0},a_0',b_0,b_0'\}$, the two neighbors of $c_0$ on $C[0,1]$ both lie in $Q[0]$ and the two neighbors of $c_1$ on $C[0,1]$ both lie in $Q[1]$. So $c_0'\in V(Q[0])$ and $c_1''\in V(Q[1])$. Clearly, $c_2''$ and $c_{k-1}'$ are neighbors of $c_2$ and $c_{k-1}$ in $Q[2]$ and $Q[k-1]$, respectively, and whenever $k$ is even, $k-1-2$ is odd. By Claim A, there is a spanning path $P_{c_2'',c_{k-1}'}$ passing through $M[2,k-1]$ in $Q[2,k-1]$. Hence, $C[0,1]+P_{c_2'',c_{k-1}'}+\{c_0c_{1},c_1''c_2'',c_0'c_{k-1}'\}-\{c_0c_0',c_1c_1''\}$ is the desired Hamiltonian cycle in $Q_n^{k}$; see Figure 4.

\begin{figure}[h]
\begin{center}
\includegraphics[scale=0.5]{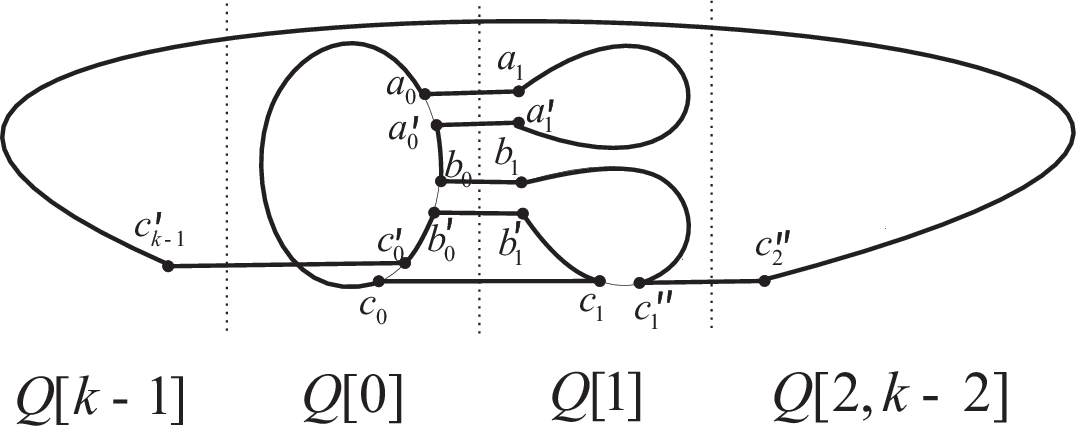}\\{Figure 4. Illustration for the proof of Subcase 3.1.2 in Claim 1.}
\end{center}
\end{figure}

\textbf{Subcase 3.2.} $1\leq j\leq k-2$.

First, we find a Hamiltonian cycle $C[0,j-1]$ containing $M[0,j-1]$ in $Q[0,j-1]$. If $j=1$, let $C[0,j-1]=C_0$. If $j>1$, since $|E(C_0)|-|M_{0}|-|M_{1}|\geq k^{n-1}-(4n-20-3)>1$, we may choose an edge $r_{0}r_{0}'\in E(C_0)\setminus M_{0}$ such that $r_{1}r_{1}'\notin M_{1}$. Since $|M_i|<2n-4$ for every $i\geq1$, Lemma \ref{hamilpath2} implies that there is a spanning path $P_{r_{1},r_{1}'}$ passing through $M[1,j-1]$ in $Q[1,j-1]$ such that $E(P_{r_1,r_1'})\cap E(Q[j-1])$ forms a spanning path in $Q[j-1]$. Let $C[0,j-1]=C_0+P_{r_{1},r_{1}'}+\{r_{0}r_{1},r_{0}'r_{1}'\}-\{r_{0}r_{0}'\}$.

Next, we find a Hamiltonian cycle $C[0,j+1]$ containing $M[0,j+1]$ in $Q[0,j+1]$. Since $a_{j}$ has $2(n-1)>3$ neighbors in $Q[j]$, we may choose one, denoted by $a_{j}'$, such that $a_{j}'\notin\{b_{j},c_{j}\}$. Since $b_{j}$ has $2(n-1)>5$ neighbors in $Q[j]$, we may choose one, denoted by $b_{j}'$, such that $b_{j}'\notin\{a_{j},a_{j}',c_{j}\}$ and $a_{j+1}'b_{j+1}'\notin M_{j+1}$. Since $c_{j}$ has $2(n-1)>8$ neighbors in $Q[j]$, we may choose one, denoted by $c_{j}'$, such that $c_{j}'\notin\{a_{j},a_{j}',b_{j},b_{j}'\}$ and $\{a_{j+1}'c_{j+1}',b_{j+1}'c_{j+1}'\}\cap M_{j+1}=\emptyset$. So $a_{j},a_{j}',b_{j},b_{j}',c_{j},c_{j}'$ are distinct vertices. Since $\{a_{j},b_{j},c_{j}\}\cap V(M_{j})=\emptyset$, we have $M_{j}\cup\{a_{j}a_{j}',b_{j}b_{j}',c_{j}c_{j}'\}$ is a linear forest. Since $|E(C[0,j-1])\cap E(Q[j-1])\setminus M_{j-1}|-2\cdot2\cdot|M_{j}\cup\{a_{j}a_{j}',b_{j}b_{j}',c_{j}c_{j}'\}|\leq k^{n-1}-1-2n-12-4(2n-12+3)>1$, there exists an edge $w_{j-1}w_{j-1}'\in E(C[0,j-1])\cap E(Q[j-1])\setminus M_{j-1}$ such that $\{w_{j},w_{j}'\}\cap V(M_{j}\cup\{a_{j}a_{j}',b_{j}b_{j}',c_{j}c_{j}'\})=\emptyset$. Thus $M_{j}\cup\{a_{j}a_{j}',b_{j}b_{j}',c_{j}c_{j}',w_{j}w_{j}'\}$ is a linear forest. Note that $|M_{j}\cup\{w_{j}w_{j}',a_{j}a_{j}',b_{j}b_{j}',c_{j}c_{j}'\}|<2(n-1)-1$ and $|M_{j+1}|\leq\lfloor\frac{4n-20-3}{2}\rfloor=2n-12=2(n-1)-10$. By Theorem \ref{forest} there is a Hamiltonian cycle $C_j$ containing $M_{j}\cup\{w_{j}w_{j}',a_{j}a_{j}',b_{j}b_{j}',c_{j}c_{j}'\}$ in $Q[j]$ and by Lemma \ref{pathpartition11} there is a spanning 3-path $P_{a_{j+1},a_{j+1}'}+P_{b_{j+1},b_{j+1}'}+P_{c_{j+1},c_{j+1}'}$ passing through $M_{j+1}$ in $Q[j+1]$. Let $C[0,j+1]=C[0,j-1]+C_j+$
$P_{a_{j+1},a_{j+1}'}+P_{b_{j+1},b_{j+1}'}+P_{c_{j+1},c_{j+1}'}+\{w_{j-1}w_{j},w_{j-1}'w_{j}',$
$a_{j}a_{j+1},a_{j}'a_{j+1}',b_{j}b_{j+1},b_{j}'b_{j+1}',c_{j}c_{j+1},c_{j}'c_{j+1}'\}-\{w_{j-1}w_{j-1}',w_{j}w_{j}',a_{j}a_{j}',b_{j}b_{j}',c_{j}c_{j}'\}$.

If $j=k-2$, then $C[0,j+1]$ is the desired Hamiltonian cycle in $Q_n^{k}$. If $j<k-2$, since $|E(C[0,j+1])\cap E(Q[j+1])|-|M_{j+1}|-|M_{j+2}|\geq k^{n-1}-3-(4n-23)>1$, there exists an edge $t_{j+1}t_{j+1}'\in E(C[0,j+1])\cap E(Q[j+1])\setminus M_{j+1}$ such that $t_{j+2}t_{j+2}'\notin M_{j+2}$. Recall that $M\cap E_d(i,i+1)=\emptyset$ and $|M_i|<2n-4$ for every $i\geq j$. Lemma \ref{hamilpath2} implies that there is a spanning path $P_{t_{j+2},t_{j+2}'}$ passing through $M[j+2,k-1]$ in $Q[j+2,k-1]$. Now $C[0,j+1]+P_{t_{j+2},t_{j+2}'}+\{t_{j+1}t_{j+2},t_{j+1}'t_{j+2}'\}-\{t_{j+1}t_{j+1}'\}$ is the desired Hamiltonian cycle in $Q_n^{k}$.

Claim 1 is proved.

\vskip 0.2cm

\textbf{Claim 2.} \emph{Let $x_0y_0\in M_0$ and $C_0$ be a Hamiltonian cycle containing $M_0\setminus\{x_0y_0\}$ in $Q[0]$. If $x_0y_0\notin E(C_0)$ and $|M\cap E_d(0,1)|+|M\cap E_d(k-1,0)|+|M[1,k-1]|\leq3$, then we can construct a Hamiltonian cycle containing $M$ in $Q_n^{k}$.}

\vskip 0.2cm

Let $x_0'$ and $y_0'$ be neighbors of $x_0$ and $y_0$ on $C_0$, respectively, such that one of the paths joining $x_0$ and $y_0$ on $C_0$ contains $x_0'$ and the other contains $y_0'$. Note that $x_0',y_0'$ are distinct, and $p(x_0')\neq p(y_0')$ when $k$ is even. There are always two ways to choose $x_0'$ and $y_0'$. The four vertices obtained from two choices are distinct, because there is no cycle of length 3 in $Q_n^{k}$. Since $x_0y_0\in M$, we have $\{x_0x_0',y_0y_0'\}\cap M=\emptyset$.

\textbf{Case 1.} $|M\cap E_d(0,1)|+|M\cap E_d(k-1,0)|=0$. Now $|M[1,k-1]|\leq3$.

\textbf{Claim B.} \emph{If $|M_1|\leq1$ and there is an integer $j(1\leq j\leq k-1)$ such that $|M\cap E_d(i,i+1)|\leq1$ for every $i\in\{1,\ldots,j-1\}$, then there is a Hamiltonian cycle $C[0,j]$ containing $M[0,j]$ in $Q[0,j]$ such that $E(C[0,j])\cap E(Q[j])$ forms a spanning path in $Q[j]$.}

Since $|M_1|\leq1$, we can choose $x_0'$ and $y_0'$ such that $x_1'y_1'\notin M_1$. Note that $x_1',y_1'$ are distinct, and $p(x_1')\neq p(y_1')$ whenever $k$ is even. By Lemma \ref{pathpartition8} there is a spanning path $P_{x_1',y_1'}$ passing through $M_1$ in $Q[1]$. Let $C[0,1]=C_0+P_{x_1',y_1'}+\{x_0y_0,x_0'x_1',y_0'y_1'\}-\{x_0x_0',y_0y_0'\}$. If $j=1$, $C[0,j]=C[0,1]$. Otherwise, $j>1$. If $M\cap E_d(1,2)=\emptyset$, since $|E(P_{x_1',y_1'})|-|M_{1}|-|M_{2}|\geq k^{n-1}-1-3>1$, we can choose an edge $r_{1}r_{1}'\in E(P_{x_1',y_1'})\setminus M_{1}$ such that $r_{2}r_{2}'\notin M_{2}$. If $|M\cap E_d(1,2)|=1$, let $M\cap E_d(1,2)=\{r_{1}r_{2}\}$ and $r_{1}'$ be a neighbor of $r_{1}$ on $P_{x_1',y_1'}$, so $\{r_{1}r_{1}',r_{2}r_{2}'\}\cap M=\emptyset$. Lemma \ref{hamilpath2} implies that there is a spanning path $P_{r_{2},r_{2}'}$ passing through $M[2,j]$ in $Q[2,j]$ such that $E(P_{r_{2},r_{2}'})\cap E(Q[j])$ forms a spanning path in $Q[j]$. Let $C[0,j]=C[0,1]+P_{r_{2},r_{2}'}+\{r_{1}r_{2},r_{1}'r_{2}'\}-\{r_{1}r_{1}'\}$; see Figure 5.

Claim B is proved.

\textbf{Subcase 1.1.} $\forall i\in\{1,\ldots,k-2\}$, $|M\cap E_d(i,i+1)|\leq1$.

Since $|M_1|+|M_{k-1}|\leq|M[1,k-1]|\leq3$, then $|M_1|\leq1$ or $|M_{k-1}|\leq1$. By symmetry we may assume $|M_1|\leq1$. By Claim B there is a Hamiltonian cycle $C[0,j]$ containing $M[0,j]$ in $Q[0,j]$. Let $j=k-1$, $C[0,j]$ is the desired Hamiltonian cycle in $Q_n^{k}$.

\textbf{Subcase 1.2.} $\exists j\in\{1,\ldots,k-2\}$, $|M\cap E_d(j,j+1)|=2$.

Since $|M[1,j]|+|M[j+1,k-1]|\leq|M[1,k-1]|-|M\cap E_d(j,j+1)|\leq3-2\leq1$, by symmetry we may assume that $|M[j+1,k-1]|=0$. Since $|M[1,j]|\leq1$, we have $|M_1|\leq1$ and $|M\cap E_d(i,i+1)|\leq1$ for every $i\in\{1,\ldots,j-1\}$. By Claim B there is a Hamiltonian cycle $C[0,j]$ containing $M[0,j]$ in $Q[0,j]$ such that $E(C[0,j])\cap E(Q[j])$ forms a spanning path in $Q[j]$. Let $M\cap E_d(j,j+1)=\{a_{j}a_{j+1},b_{j}b_{j+1}\}$. If $d_{C[0,j]\cap E(Q[j])}(a_j,b_j)=1$, since $\{a_ja_{j+1},b_jb_{j+1}\}\subseteq M$, we have $\{a_jb_j,a_{j+1}b_{j+1}\}\cap M=\emptyset$. Lemma \ref{hamilpath2} implies that there is a spanning path $P_{a_{j+1},b_{j+1}}$ in $Q[j+1,k-1]$. Hence, $C[0,j]+P_{a_{j+1},b_{j+1}}+\{a_ja_{j+1},b_jb_{j+1}\}-\{a_jb_j\}$ is the desired Hamiltonian cycle in $Q_n^{k}$. Otherwise, $d_{C[0,j]\cap E(Q[j])}(a_j,b_j)>1$. Choose neighbors $a_j'$ of $a_j$ and $b_j'$ of $b_j$ on $C[0,j]\cap E(Q[j])$, respectively, such that $a_j'\neq b_j'$. Note that $a_{j+1},a_{j+1}',b_{j+1},b_{j+1}'$ are distinct vertices in $Q[j+1]$ and whenever $k$ is even, $p(a_{j+1})\neq p(a_{j+1}')$ and $p(b_{j+1})\neq p(b_{j+1}')$. By Lemma \ref{pathpartition3} there is a spanning 2-path $P_{a_{j+1},a_{j+1}'}+P_{b_{j+1},b_{j+1}'}$ in $Q[j+1,k-1]$. Now $C[0,j]+P_{a_{j+1},a_{j+1}'}+P_{b_{j+1},b_{j+1}'}+\{a_ja_{j+1},a_j'a_{j+1}',b_jb_{j+1},b_j'b_{j+1}'\}-\{a_jb_j',b_jb_j'\}$ is the desired Hamiltonian cycle in $Q_n^{k}$; see Figure 5.

\begin{figure}[h]
\begin{center}
\includegraphics[scale=0.5]{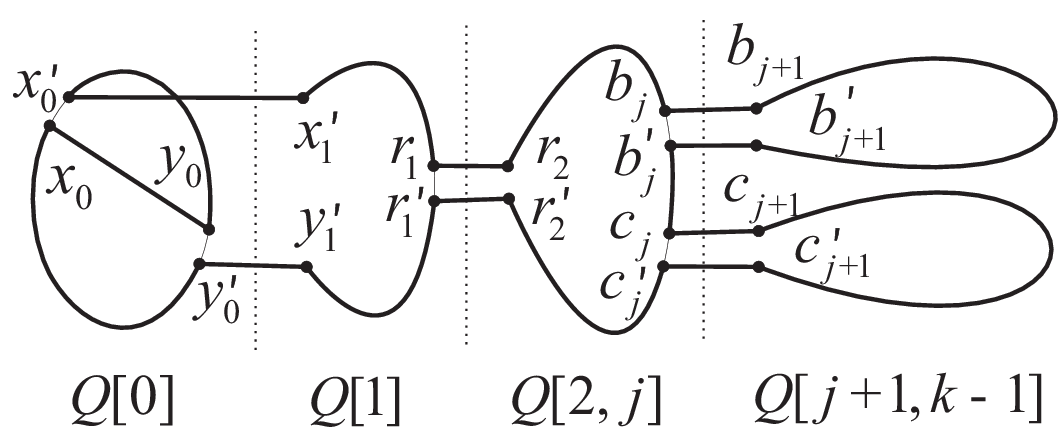}\\{Figure 5. Illustration for the proof of Subcase 1.2 in Claim 2.}
\end{center}
\end{figure}

\textbf{Subcase 1.3.} $\exists j\in\{1,\ldots,k-2\}$, $|M\cap E_d(j,j+1)|=3$.

Now $M[1,j]=M[j+1,k-1]=\emptyset$. Since $k\geq4$, we have $j>1$ or $j<k-2$. By symmetry, we may assume $j>1$. By Claim B there is a Hamiltonian cycle $C[0,j-1]$ containing $M[0,j-1]$ in $Q[0,j-1]$ such that $E(C[0,j-1])\cap E(Q[j-1])$ forms a spanning path in $Q[j-1]$. Let $M\cap E_d(j,j+1)=\{a_{j}a_{j+1},b_{j}b_{j+1},c_{j}c_{j+1}\}$. Since $|E(C[0,j-1])\cap E(Q[j-1])|\geq k^{n-1}-1>7$ and at most six edges in $E(C[0,j-1])\cap E(Q[j-1])$ are incident with $\{a_{j},b_{j},c_{j}\}$, there exists an edge $w_{j-1}w_{j-1}'\in E(C[0,j-1])\cap E(Q[j-1])$ such that $\{w_{j},w_{j}'\}\cap\{a_{j},b_{j},c_{j}\}$. Since $a_{j},b_{j}$ and $c_{j}$ all have $2(n-1)>7$ neighbors on $Q[j]$, we may choose neighbors $a_{j}',b_{j}',c_{j}'$ of $a_{j},b_{j},c_{j}$ on $Q[j]$, respectively, such that $\{a_{j}',b_{j}',c_{j}'\}\cap\{w_{j},w_{j}',a_{j},b_{j},c_{j}\}=\emptyset$ and $a_{j}',b_{j}',c_{j}'$ are distinct. Note that $w_{j},w_{j}',a_{j},a_{j}',b_{j},b_{j}',c_{j},c_{j}'$ are distinct vertices, thus $\{w_{j}w_{j}',a_{j}a_{j}',b_{j}b_{j}',c_{j}c_{j}'\}$ is a matching of size 4 in $Q[j]$. By Theorem \ref{forest} there is a Hamiltonian cycle $C_j$ containing $\{w_{j}w_{j}',a_{j}a_{j}',b_{j}b_{j}',c_{j}c_{j}'\}$ in $Q[j]$ and by Lemma \ref{pathpartition11} there is a spanning 3-path $P_{a_{j+1},a_{j+1}'}+P_{b_{j+1},b_{j+1}'}+P_{c_{j+1},c_{j+1}'}$ in $Q[j+1]$. Let $C[0,j+1]=C[0,j-1]+C_j+$
$P_{a_{j+1},a_{j+1}'}+P_{b_{j+1},b_{j+1}'}+P_{c_{j+1},c_{j+1}'}+\{w_{j-1}w_{j},w_{j-1}'w_{j}',a_{j}a_{j+1},a_{j}'a_{j+1}',$
$b_{j}b_{j+1},b_{j}'b_{j+1}',c_{j}c_{j+1},c_{j}'c_{j+1}'\}-\{w_{j-1}w_{j-1}',w_{j}w_{j}',a_{j}a_{j}',b_{j}b_{j}',c_{j}c_{j}'\}$.

If $j=k-2$, then $C[0,j+1]$ is the desired Hamiltonian cycle in $Q_n^{k}$. If $j<k-2$, we choose an edge $t_{j+1}t_{j+1}'\in E(C[0,j+1])\cap E(Q[j+1])$. Lemma \ref{hamilpath2} implies that there is a spanning path $P_{t_{j+2},t_{j+2}'}$ in $Q[j+2,k-1]$. Now $C[0,j+1]+P_{t_{j+2},t_{j+2}'}+\{t_{j+1}t_{j+2},t_{j+1}'t_{j+2}'\}-\{t_{j+1}t_{j+1}'\}$ is the desired Hamiltonian cycle in $Q_n^{k}$.

\textbf{Case 2.} $|M\cap E_d(0,1)|+|M\cap E_d(k-1,0)|=1$. Now $|M[1,k-1]|\leq2$.

By symmetry, we may assume $M\cap E_d(0,1)=\{a_0a_1\}$ and $M\cap E_d(k-1,0)=\emptyset$. If $d_{C_0}(a_0,x_0y_0)=1$, we choose $x_0'$ and $y_0'$ such that $a_0\in\{x_0',y_0'\}$. Now $x_1'y_1'\notin M_1$ and $|M_{1}|\leq2$, which also could use the Lemma \ref{pathpartition8} to construct the cycle $C[0,j]$ in Claim B. So this situation could return to the Case 1. Next, we discuss $d_{C_0}(a_0,x_0y_0)>1$. Note that the two ways to choose $x_0'$ and $y_0'$ both satisfy $a_0\notin\{x_0',y_0'\}$.

\textbf{Subcase 2.1.} $|M_{1}|=0$.

Choose a neighbor $a_0'$ of $a_0$ on $C_0$ such that $a_0'\notin\{x_0',y_0'\}$. Note that $a_1,a_1',x_1',y_1'$ are distinct and whenever $k$ is even, $p(x_1')\neq p(y_1')$. By Lemma \ref{pathpartition9} there is a spanning 2-path $P_{x_1',y_1'}+P_{a_1,a_1'}$ in $Q[1]$, where $P_{a_1,a_1'}=a_1a_1'$ is an edge.

If $M\cap E_d(1,2)=\emptyset$, since $|E(P_{x_1',y_1'}+a_1a_1')|-|M_{2}|\geq k^{n-1}-2-2>1$, we can choose an edge $w_1w_1'\in E(P_{x_1',y_1'}+a_1a_1')$ such that $w_2w_2'\notin M_{2}$. If $|M\cap E_d(1,2)|=1$, then let $M\cap E_d(1,2)=\{w_1w_2\}$ and $w_1'$ be a neighbor of $w_1$ on $P_{x_1',y_1'}+a_1a_1'$. If $|M\cap E_d(1,2)|=2$ (now denote $M\cap E_d(1,2)=\{b_1b_2,c_1c_2\}$) and $d_{P_{x_1',y_1'}+a_1a_1'}(b_1,c_1)=1$, then let $w_1w_1'=b_1c_1$. In the above three cases, we all have $w_2w_2'\notin M_{2}$. Since $|M[2,k-1]|\leq2$, by Lemma \ref{pathpartition8} there is a spanning path $P_{w_2,w_2'}$ passing through $M[2,k-1]$ in $Q[2,k-1]$. Hence, $C_0+P_{x_1',y_1'}+P_{w_2,w_2'}+\{x_0y_0,x_0'x_{1}',y_0'y_{1}',a_0a_1,a_0'a_1',a_1a_1',w_1w_2,w_1'w_2'\}-\{x_0x_0',y_0y_0',a_0a_0',w_1w_1'\}$ is the desired Hamiltonian cycle in $Q_n^{k}$.

It remains to consider the case $M\cap E_d(1,2)=\{b_1b_2,c_1c_2\}$ and $d_{P_{x_1',y_1'}+a_1a_1'}(b_1,c_1)>1$. Now $|M[2,k-1]|=0$. Choose neighbors $b_1',c_1'$ of $b_1,c_1$ on $P_{x_1',y_1'}+a_1a_1'$ such that $b_1'\neq c_1'$. Note that $b_2,b_2',c_2,c_2'$ are distinct and whenever $k$ is even, $p(b_2)\neq p(b_2')$ and $p(c_2)\neq p(c_2')$. By Lemma \ref{pathpartition3}, there is a spanning 2-path $P_{b_2,b_2'}+P_{c_2,c_2'}$ in $Q[2,k-1]$. Now $C_0+P_{x_1',y_1'}+P_{b_2,b_2'}+P_{c_2,c_2'}+\{x_0y_0,x_0'x_{1}',y_0'y_{1}',a_0a_1,a_0'a_1',a_1a_1',b_1b_2,b_1'b_2',c_1c_2,c_1'c_2'\}$
$-\{x_0x_0',y_0y_0',a_0a_0',b_1b_1',c_1c_1'\}$ is the desired Hamiltonian cycle in $Q_n^{k}$.

\textbf{Subcase 2.2.} $|M_{1}|\geq1$.

Since $|M[1,k-1]|\leq2$, we have $|M_{k-1}|\leq1$ and $\sum_{i=1}^{k-2}|M\cap E_d(i,i+1)|\leq1$. Since $k\geq4$, there exists an integer $j\in\{1,\ldots,k-2\}$ such that $M\cap E_d(j,j+1)=\emptyset$. Choose neighbors $a_0',x_0',y_0'$ of $a_0,x_0,y_0$ on $C_0$ such that $x_{k-1}'y_{k-1}'\notin M_{k-1}$. Since $\{x_0y_0,a_0a_1\}\subseteq M$, we have $\{x_0x_0',y_0y_0',a_0a_0',a_1a_1'\}\cap M=\emptyset$. Lemma \ref{pathpartition8} implies that there exist spanning paths $P_{a_1,a_1'}$ passing through $M[1,j]$ in $Q[1,j]$ and $P_{x_{k-1}',y_{k-1}'}$ passing through $M[j+1,k-1]$ in $Q[j+1,k-1]$, respectively. Hence, $C_0+P_{a_1,a_1'}+P_{x_{k-1}',y_{k-1}'}+\{x_0y_0,a_0a_1,a_0'a_1',x_0'x_{k-1}',y_0'y_{k-1}'\}-\{a_0a_0',x_0x_0',y_0y_0'\}$ is the desired Hamiltonian cycle in $Q_n^{k}$.

\textbf{Case 3.} $|M\cap E_d(0,1)|+|M\cap E_d(k-1,0)|=2$. Now $|M[1,k-1]|\leq1$.

Since $k\geq4$, there exists an integer $j\in\{1,\ldots,k-2\}$ such that $M\cap E_d(j,j+1)=\emptyset$. By symmetry we may assume $|M\cap E_d(k-1,0)|\leq1$.

\textbf{Subcase 3.1.} $M\cap E_d(k-1,0)=\emptyset$. Now $|M\cap E_d(0,1)|=2$.

Let $M\cap E_d(0,1)=\{a_0a_1,b_0b_1\}$. Since $|M_{k-1}|\leq1$, we can choose $x_0',y_0'$ such that $x_{k-1}'y_{k-1}'\notin M_{k-1}$. Let $P_{x_0',y_0'}=C_0+\{x_0y_0\}-\{x_0x_0',y_0y_0'\}$.

If $d_{P_{x_0',y_0'}}(a_0,b_0)=1$, then by Lemma \ref{hamilpathone} there exist spanning paths $P_{a_1,b_1}$ passing through $M[1,j]$ in $Q[1,j]$ and $P_{x_{k-1}',y_{k-1}'}$ passing through $M[j+1,k-1]$ in $Q[j+1,k-1]$, respectively. Hence, $P_{x_0',y_0'}+P_{a_1,b_1}+P_{x_{k-1}',y_{k-1}'}+\{x_0'x_{k-1}',y_0'y_{k-1}',a_0a_1,b_0b_1\}-\{a_0b_0\}$ is the desired Hamiltonian cycle in $Q_n^{k}$.

It remains to consider the case that $d_{P_{x_0',y_0'}}(a_0,b_0)>1$.

If $\{x_0',y_0'\}=\{a_0,b_0\}$, then by Lemma \ref{hamilpathone} there exists a spanning path $P_{x_{1}',y_{1}'}$ passing through $M[1,k-1]$ in $Q[1,k-1]$. Hence, $P_{x_0',y_0'}+P_{x_{1}',y_{1}'}+\{x_0'x_{1}',y_0'y_{1}'\}$ is the desired Hamiltonian cycle in $Q_n^{k}$.

If $|\{x_0',y_0'\}\cap\{a_0,b_0\}|=1$, without loss of generality we may assume $x_0'=a_0$ and $y_0'\neq b_0$. First, we choose a neighbor $b_0'$ of $b_0$ on $P_{x_0',y_0'}$ such that $b_1'\notin \{x_1',y_1'\}$ and $b_1'y_1'\notin M_{1}$. When $d_{P_{x_0',y_0'}}(b_0,y_0')>1$, the two neighbors of $b_0$ on $P_{x_0',y_0'}$ are neither $x_0'$ nor $y_0'$. We choose one, denoted by $b_0'$, such that $b_1'y_1'\notin M_{1}$. Clearly, $b_1'\notin \{x_1',y_1'\}$. When $d_{P_{x_0',y_0'}}(b_0,y_0')=1$, we choose the neighbor $b_0'$ of $b_0$ on $P_{x_0',y_0'}$ which is not $y_0'$. Clearly $b_1'\notin \{x_1',y_1'\}$. Note that $b_1'y_1'$ can not be an edge, otherwise, $b_1b_1',b_1'y_1',b_1y_1'$ forms a cycle of length 3 in $Q[0]$, which contradicts with $k\geq4$. So $b_1'y_1'\notin M_{1}$. Since $\{b_1,x_1'\}\cap V (M_{1})=\emptyset$ and $b_1'y_1'\notin M_{1}$, by Lemma \ref{pathpartition7} there is a spanning 2-path $P_{x_1',y_1'}+P_{b_1,b_1'}$ passing through $M[1,k-1]$ in $Q[1,k-1]$. Hence, $P_{x_0',y_0'}+P_{x_1',y_1'}+P_{b_1,b_1'}+\{x_0'x_{1}',y_0'y_{1}',b_0b_1,b_0'b_{1}'\}-\{b_0b_0'\}$ is the desired Hamiltonian cycle in $Q_n^{k}$.

If $\{x_0',y_0'\}\cap\{a_0,b_0\}=\emptyset$, since $d_{P_{x_0',y_0'}}(a_0,b_0)>1$ and $|M_{1}|\leq1$, we may choose neighbors $a_0',b_0'$ of $a_0,b_0$ on $C_0$ such that $a_1'b_1'\notin M_{1}$. It may happen that $\{a_0',b_0'\}\cap\{x_0',y_0'\}\neq\emptyset$, the constructions described below remain valid in these degenerate cases. Since $\{a_1,b_1\}\cap V (M_{1})=\emptyset$ and $a_1'b_1'\notin M_{1}$, by Lemma \ref{pathpartition7} when $j>1$ and Lemma \ref{pathpartition2} when $j=1$ there is a spanning 2-path $P_{a_1,a_1'}+P_{b_1,b_1'}$ passing through $M[1,j]$ in $Q[1,j]$. Next by Lemma \ref{hamilpathone} to construct a spanning path $P_{x_{k-1}',y_{k-1}'}$ passing through $M[j+1,k-1]$ in $Q[j+1,k-1]$. Now $P_{x_0',y_0'}+P_{x_{k-1}',y_{k-1}'}+P_{a_1,a_1'}+P_{b_1,b_1'}+\{x_0'x_{k-1}',y_0'y_{k-1}',a_0a_1,$
$a_0'a_1',b_0b_1,b_0'b_1'\}-\{a_0a_0',b_0b_0'\}$ is the desired Hamiltonian cycle in $Q_n^{k}$.

\textbf{Subcase 3.2.} $|M\cap E_d(k-1,0)|=1$. Now $|M\cap E_d(0,1)|=1$.

Let $M\cap E_d(0,1)=\{a_0a_1\}$ and $M\cap E_d(k-1,0)=\{b_0b_{k-1}\}$. Since there have two ways to choose $x_0',y_0'$ and the four vertices obtained from these two choices are distinct, so we can choose $x_0',y_0'$ such that $|\{x_0',y_0'\}\cap\{a_0,b_0\}|\leq1$. Let $P_{x_0',y_0'}=C_0+\{x_0y_0\}-\{x_0x_0',y_0y_0'\}$.

If $\{x_0',y_0'\}\cap\{a_0,b_0\}=\emptyset$, since $|M_{1}|+|M_{k-1}|\leq|M[1,k-1]|\leq1$, by symmetry we may assume $M_{1}=\emptyset$. Choose a neighbor $a_0'$ of $a_0$ on $P_{x_0',y_0'}$ such that $a_0'\notin\{x_0',y_0'\}$, and choose a neighbor $b_0'$ of $b_0$ on $P_{x_0',y_0'}$ such that $a_0a_0'\neq b_0b_0'$. Since $\{a_0a_1,b_0b_{k-1}\}\subseteq M$, we have $\{a_0a_0',a_1a_1',b_0b_0',b_{k-1}b_{k-1}'\}\cap M=\emptyset$. Note that $a_1,a_1',x_1',y_1'$ are distinct and whenever $k$ is even, $p(x_{1}')\neq p(y_{1}')$ and $p(a_1)\neq p(a_1')$. By Lemma \ref{pathpartition7} when $j>1$ and Lemma \ref{pathpartition2} when $j=1$ there exists a spanning 2-path $P_{a_1,a_1'}+P_{x_1',y_1'}$ passing through $M[1,j]$ in $Q[1,j]$. Lemma \ref{hamilpath2} implies that there exists a spanning path $P_{b_{k-1},b_{k-1}'}$ passing through $M[j+1,k-1]$ in $Q[j+1,k-1]$. Now $P_{x_0',y_0'}+P_{a_1,a_1'}+P_{x_1',y_1'}+P_{b_{k-1},b_{k-1}'}+\{x_0'x_{1}',y_0'y_{1}',a_0a_1,a_0'a_1',b_0b_{k-1},b_0'b_{k-1}'\}-\{a_0a_0',b_0b_0'\}$ is the desired Hamiltonian cycle in $Q_n^{k}$.

If $|\{x_0',y_0'\}\cap\{a_0,b_0\}|=1$, then by symmetry we may assume $x_0'=a_0$ and $y_0'\neq b_0$. Choose a neighbor $b_0'$ of $b_0$ on $P_{x_0',y_0'}$. Note that $x_1'y_1'\notin M_{1}$ and $b_{k-1}b_{k-1}'\notin M_{k-1}$. Lemma \ref{hamilpathone} implies that there exist spanning paths $P_{x_1',y_1'}$ passing through $M[1,j]$ in $Q[1,j]$ and $P_{b_{k-1},b_{k-1}'}$ passing through $M[j+1,k-1]$ in $Q[j+1,k-1]$, respectively. Now $P_{x_0',y_0'}+P_{x_1',y_1'}+P_{b_{k-1},b_{k-1}'}+\{x_0'x_{1}',y_0'y_{1}',b_0b_{k-1},b_0'b_{k-1}'\}-\{b_0b_0'\}$ is the desired Hamiltonian cycle in $Q_n^{k}$.

\textbf{Case 4.} $|M\cap E_d(0,1)|+|M\cap E_d(k-1,0)|=3$. Now $|M[1,k-1]|=0$.

Denote $V(M\cap E_d)\cap V(Q[0])=\{a_0,b_0,c_0\}$. Recall that there have two ways to choose $x_0'$ and $y_0'$, we choose $x_0',y_0'$ such that $|\{x_0',y_0'\}\cap\{a_0,b_0,c_0\}|\leq1$. Let $P_{x_0',y_0'}=C_0+\{x_0y_0\}-\{x_0x_0',y_0y_0'\}$.

\textbf{Subcase 4.1.} $\{x_0',y_0'\}\cap\{a_0,b_0,c_0\}=\emptyset$. By symmetry, there are two cases to consider.

\textbf{Subcase 4.1.1.} $M\cap E_d=\{a_0a_1,b_0b_1,c_0c_{1}\}$.

If $d_{P_{x_0',y_0'}}(a_0,b_0)=1$, or $d_{P_{x_0',y_0'}}(a_0,c_0)=1$, or $d_{P_{x_0',y_0'}}(b_0,c_0)=1$, without loss of generality assume $d_{P_{x_0',y_0'}}(a_0,b_0)=1$. Choose a neighbor $c_0'$ of $c_0$ on $P_{x_0',y_0'}$ such that $c_0'\notin\{a_0,b_0\}$. By Lemma \ref{pathpartition3} and Lemma \ref{hamilpath1} there exist spanning 2-path $P_{a_1,b_1}+P_{c_{1},c_{1}'}$ in $Q[1,k-2]$ and spanning path $P_{x_{k-1}',y_{k-1}'}$ in $Q[k-1]$, respectively. Now $P_{x_0',y_0'}+P_{a_1,b_1}+P_{c_{1},c_{1}'}+P_{x_{k-1}',y_{k-1}'}+\{x_0'x_{k-1}',y_0'y_{k-1}',a_0a_{1},b_0b_{1},c_0c_{1},c_0'c_{1}'\}$
$-\{a_0b_0,c_0c_0'\}$ is the desired Hamiltonian cycle in $Q_n^{k}$. If $d_{P_{x_0',y_0'}}(a_0,b_0)>1$, $d_{P_{x_0',y_0'}}(a_0,c_0)>1$ and $d_{P_{x_0',y_0'}}(b_0,c_0)>1$, then we choose neighbors $a_0'$, $b_0'$ and $c_0'$ of $a_0$, $b_0$ and $c_0$ on $P_{x_0',y_0'}$, respectively, in the same direction on $P_{x_0',y_0'}$. Note that $a_0,a_0',b_0,b_0',c_0,c_0'$ are distinct. By Lemma \ref{pathpartition11} and Lemma \ref{hamilpath1} there exist spanning 3-path $P_{a_1,a_1'}+P_{b_1,b_1'}+P_{c_1,c_1'}$ in $Q[1]$ and spanning path $P_{x_{k-1}',y_{k-1}'}$ in $Q[2,k-1]$, respectively. Hence, $P_{x_0',y_0'}+P_{a_1,a_1'}+P_{b_1,b_1'}+P_{c_1,c_1'}+P_{x_{k-1}',y_{k-1}'}+\{x_0'x_{k-1}',y_0'y_{k-1}',a_0a_{1},a_0'a_{1}',b_0b_{1},b_0'b_{1}',$
$c_0c_{1},c_0'c_{1}'\}-\{a_0a_0',b_0b_0',c_0c_0'\}$ is the desired Hamiltonian cycle in $Q_n^{k}$.

\textbf{Subcase 4.1.2.} $M\cap E_d=\{a_0a_1,b_0b_1,c_0c_{k-1}\}$.

If $d_{P_{x_0',y_0'}}(a_0,b_0)=1$, we choose a neighbor $c_0'$ of $P_{x_0',y_0'}$ on $C_0$ such that $c_0'\notin\{x_0',y_0'\}$. By Lemma \ref{pathpartition3} and Lemma \ref{hamilpath2} there exist spanning 2-path $P_{c_{k-1},c_{k-1}'}+P_{x_{k-1}',y_{k-1}'}$ in $Q[k-1]$ and spanning path $P_{a_1,b_1}$ in $Q[1,k-2]$. Now $P_{x_0',y_0'}+P_{a_1,b_1}+P_{c_{k-1},c_{k-1}'}+P_{x_{k-1}',y_{k-1}'}+\{x_0'x_{k-1}',y_0'y_{k-1}',a_0a_{1},b_0b_{1},c_0c_{k-1},c_0'c_{k-1}'\}$
$-\{a_0b_0,c_0c_0'\}$ is the desired Hamiltonian cycle in $Q_n^{k}$. If $d_{P_{x_0',y_0'}}(a_0,b_0)>1$, we choose a neighbor $c_0'$ of $c_0$ on $P_{x_0',y_0'}$ such that $c_0'\notin\{x_0',y_0'\}$, and choose neighbors $a_0'$, $b_0'$ of $a_0$, $b_0$ in the same direction on $P_{x_0',y_0'}$ such that $c_0c_0'\notin\{a_0a_0',b_0b_0'\}$. Note that $a_1,a_1',b_1,b_1',c_{k-1},c_{k-1}',$
$x_{k-1}',y_{k-1}'$ are distinct and whenever $k$ is even, $p(a_1)\neq p(a_1')$, $p(b_1)\neq p(b_1')$, $p(c_{k-1})\neq p(c_{k-1}')$ and $p(x_{k-1}')\neq p(y_{k-1}')$. By Lemma \ref{pathpartition3} there exist spanning 2-paths $P_{a_1,a_1'}+P_{b_1,b_1'}$ in $Q[1,k-2]$ and $P_{c_{k-1},c_{k-1}'}+P_{x_{k-1}',y_{k-1}'}$ in $Q[k-1]$, respectively. Hence, $P_{x_0',y_0'}+P_{a_1,a_1'}+P_{b_1,b_1'}+P_{c_{k-1},c_{k-1}'}+P_{x_{k-1}',y_{k-1}'}+\{x_0'x_{k-1}',y_0'y_{k-1}',a_0a_{1},a_0'a_{1}',b_0b_{1},b_0'b_{1}',$
$c_0c_{k-1},c_0'c_{k-1}'\}-\{a_0a_0',b_0b_0',c_0c_0'\}$ is the desired Hamiltonian cycle in $Q_n^{k}$.

\textbf{Subcase 4.2.} $|\{x_0',y_0'\}\cap\{a_0,b_0,c_0\}|=1$.

Without loss of generality, assume $\{x_0',y_0'\}\cap\{a_0,b_0,c_0\}=\{a_0\}$ and $a_0a_1\in M\cap E_d$. We distinguish three cases to consider.

\textbf{Subcase 4.2.1.} $\{b_0b_{k-1},c_{0}c_{k-1}\}\subseteq M\cap E_d$.

Choose the neighbors $b_0',c_0'$ of $b_0,c_0$ on $P_{x_0',y_0'}-E(P_{x_0',y_0'}[b_0,c_0])$, respectively. Then $b_0,b_0',c_0,c_0'$ are distinct. By Lemma \ref{pathpartition3} and Lemma \ref{hamilpath1} there exist spanning 2-path $P_{b_{k-1},b_{k-1}'}+P_{c_{k-1},c_{k-1}'}$ in $Q[k-1]$ and spanning path $P_{x_{1}',y_{1}'}$ in $Q[1,k-2]$. Hence, $P_{x_0',y_0'}+P_{b_{k-1},b_{k-1}'}+P_{c_{k-1},c_{k-1}'}+P_{x_{1}',y_{1}'}+\{x_0'x_{1}',y_0'y_{1}',$
$b_0b_{k-1},b_0'b_{k-1}',c_0c_{k-1},c_0'c_{k-1}'\}-\{b_0b_0',c_0c_0'\}$ is the desired Hamiltonian cycle in $Q_n^{k}$.

\textbf{Subcase 4.2.2.} $\{b_0b_{1},c_{0}c_{k-1}\}\subseteq M\cap E_d$.

Choose neighbors $b_0',c_0'$ of $b_0,c_0$ on $P_{x_0',y_0'}$, respectively, such that $b_{0}'\notin \{x_0',y_0'\}$ and $c_0c_{0}'\neq b_0b_{0}'$. Note that $x_{1}',y_{1}',b_1,b_1',c_{k-1},c_{k-1}'$ are distinct. By Lemma \ref{pathpartition3} and Lemma \ref{hamilpath1} there exist spanning 2-path $P_{x_{1}',y_{1}'}+P_{b_1,b_1'}$ in $Q[1,k-2]$ and spanning path $P_{c_{k-1},c_{k-1}'}$ in $Q[k-1]$. Hence, $P_{x_0',y_0'}+P_{x_{1}',y_{1}'}+P_{b_1,b_1'}+P_{c_{k-1},c_{k-1}'}+\{x_0'x_{1}',y_0'y_{1}',b_0b_{1},b_0'b_{1}',c_0c_{k-1},c_0'c_{k-1}'\}$
$-\{b_0b_0',c_0c_0'\}$ is the desired Hamiltonian cycle in $Q_n^{k}$.

\textbf{Subcase 4.2.3.} $\{b_0b_{1},c_{0}c_{1}\}\subseteq M\cap E_d$.

Note that at least one of $b_0$ and $c_{0}$, say $b_0$, has a neighbor $b_0'$ which is different from $x_0',y_0',b_0$ and $c_{0}$. Lemma \ref{pathpartition9} implies that there is a spanning path $P_{x_{1}',y_{1}'}$ in $Q[1]-\{b_1,b_1'\}$. Note that $c_1$ lies on $P_{x_{1}',y_{1}'}$. Choose the neighbor $c_0'$ of $c_0$ which is closer to $x_0'$ on $P_{x_0',y_0'}$, and choose the neighbor $c_1''$ of $c_1$ which is closer to $y_{1}'$ on $P_{x_{1}',y_{1}'}$. Since $c_0\notin\{b_0,b_0'\}$, we have $c_0c_0'\neq b_0b_0'$. Note that $k-1\neq2$ and $p(c_{k-1}')=p(c_{0})\neq p(c_{1})=p(c_2'')$ when $k$ is even. By Lemma \ref{hamilpath1} there is a spanning path $P_{c_2'',c_{k-1}'}$ in $Q[2,k-1]$. Now $P_{x_0',y_0'}+P_{x_{1}',y_{1}'}+P_{c_2'',c_{k-1}'}+\{x_0'x_{1}',y_0'y_{1}',b_0b_{1},b_0'b_{1}',c_0c_{1},c_1''c_2'',c_0'c_{k-1}'\}$
$-\{b_0b_0',c_0c_0',c_1c_1''\}$ is the desired Hamiltonian cycle in $Q_n^{k}$; see Figure 6.

\begin{figure}[h]
\begin{center}
\includegraphics[scale=0.5]{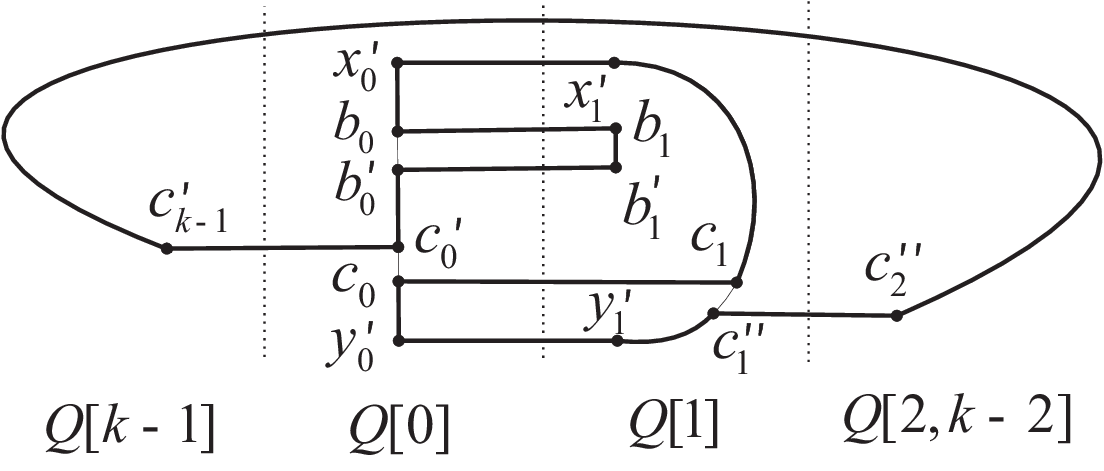}\\{Figure 6. Illustration for the proof of Subcase 4.2.3 in Claim 2.}
\end{center}
\end{figure}

Claim 2 is proved.

\vskip 0.2cm

\textbf{Claim 3.} \emph{Let $\{x_0y_0,m_0n_0\}\subseteq M_0$ and $C_0$ be a Hamiltonian cycle containing $M_0\setminus\{x_0y_0,m_0n_0\}$ in $Q[0]$. If $\{x_0y_0,m_0n_0\}\cap E(C_0)=\emptyset$ and $|M\cap E_d(0,1)|+|M\cap E_d(k-1,0)|+|M[1,k-1]|\leq2$, then we can construct a Hamiltonian cycle containing $M$ in $Q_n^{k}$.}

\vskip 0.2cm

Since $\{x_0y_0,m_0n_0\}\cap E(C_0)=\emptyset$, either one of the paths on $C_0$ joining $x_0$ and $y_0$ contains both $m_0$ and $n_0$, or one of the paths on $C_0$ joining $x_0$ and $y_0$ contains $m_0$ and the other contains $n_0$; see Figure 7 for example. Note that there are two directions, clockwise and counterclockwise, along a cycle. We choose clockwise neighbors $x_0',y_0',m_0',n_0'$ of $x_0,y_0,m_0,n_0$ on $C_{0}$, respectively; see Figure 7. Since $M$ is a matching and $\{x_0y_0,m_0n_0\}\subseteq M_0$, we have $\{x_0x_0',y_0y_0',m_0m_0',n_0n_0'\}\cap M=\emptyset$, and $x_0',y_0',m_0',n_0'$ are distinct, and whenever $k$ is even, $p(x_0')\neq p(y_0')$ and $p(m_0')\neq p(n_0')$. It may happen that $x_0'=m_0$ or $n_0'=y_0$ in Figure 7(1), or even some of $x_0'=m_0,y_0'=n_0,m_0'=y_0,n_0'=x_0$ in Figure 7(2).

Regardless of the above cases, $C_0-\{x_0x_0',y_0y_0',m_0m_0',n_0n_0'\}+\{x_0y_0,m_0n_0\}$ is a spanning 2-path passing through $M_0$ in $Q[0]$, denoted by $P_{c_0,d_0}+P_{e_0,f_0}$. Then $\{c_0,d_0,e_0,f_0\}=\{x_0',y_0',m_0',n_0'\}$, which is balanced when $k$ is even. First, choose a vertex in $\{c_0,d_0,e_0,f_0\}$ randomly, say $d_0$, then there exists at least one vertex in $\{e_0,f_0\}$, say $e_0$, satisfying $p(e_0)\neq p(d_0)$ when $k$ is even. If not, $p(d_0)=p(e_0)=p(f_0)$, this is contradictory. Therefore, $c_0,d_0,e_0,f_0$ are distinct, and whenever $k$ is even, $p(e_0)\neq p(d_0)$ and $p(c_0)\neq p(f_0)$.

\begin{figure}[h]
\begin{center}
\includegraphics[scale=0.5]{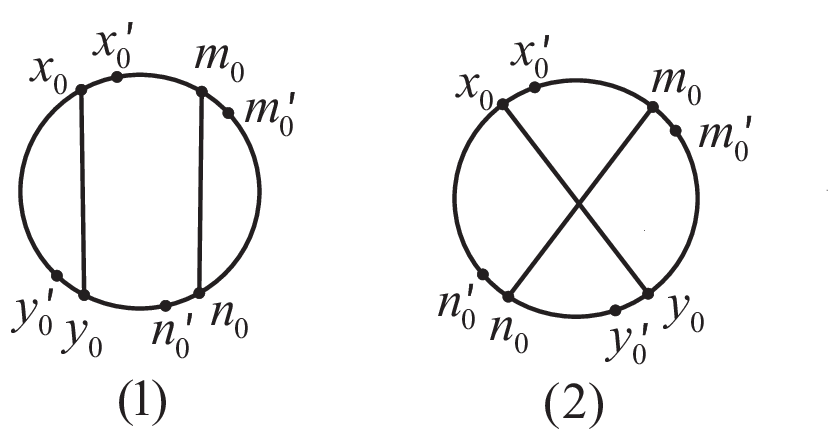}\\{Figure 7. Two possibilities of $\{x_0y_0,m_0n_0\}$ and neighbors of $x_0,y_0,m_0,n_0$ in Claim 3.}
\end{center}
\end{figure}

\textbf{Case 1.} $|M\cap E_d(0,1)|+|M\cap E_d(k-1,0)|=0$. Now $|M[1,k-1]|\leq2$.

\textbf{Subcase 1.1.} There exists an integer $j\in\{1,\ldots,k-2\}$ such that $M\cap E_d(j,j+1)=\emptyset$.

Since $|M_{1}|+|M_{k-1}|\leq|M[1,k-1]|\leq2$, by symmetry we can assume $|M_{1}|\leq1$. So $c_1f_1\notin M_1$ or $d_1e_1\notin M_1$. Without loss of generality, we assume $c_1f_1\notin M_1$. Note that $c_1,f_1$ are distinct and whenever $k$ is even, $p(c_1)\neq p(f_1)$. By Lemma \ref{pathpartition8} there is a spanning path $P_{c_1,f_1}$ passing through $M_1$ in $Q[1,j]$; see Figure 8.

\textbf{Subcase 1.1.1.} $d_{k-1}e_{k-1}\notin M_{k-1}$.

Note that $d_{k-1},e_{k-1}$ are distinct and whenever $k$ is even, $p(d_{k-1})\neq p(e_{k-1})$. By Lemma \ref{pathpartition8} there is a spanning path $P_{d_{k-1},e_{k-1}}$ passing through $M[j+1,k-1]$ in $Q[j+1,k-1]$. Hence, $P_{c_0,d_0}+P_{e_0,f_0}+P_{c_1,f_1}+P_{d_{k-1},e_{k-1}}+\{c_0c_1,f_0f_1,d_0d_{k-1},e_0e_{k-1}\}$ is the desired Hamiltonian cycle in $Q_n^{k}$.

\textbf{Subcase 1.1.2.} $d_{k-1}e_{k-1}\in M_{k-1}$.

Since $|E(P_{c_0,d_0}+P_{e_0,f_0})|-|M_{0}|-|M_{k-1}|\geq k^{n-1}-2-(4n-20)>5$, we may choose an edge $r_{0}r_{0}'\in E(P_{c_0,d_0}+P_{e_0,f_0})\setminus M_{0}$ such that $r_{k-1}r_{k-1}'\notin M_{k-1}$ and $\{r_{k-1},r_{k-1}'\}\cap \{d_{k-1},e_{k-1}\}=\emptyset$. Since $|M_{k-1}\setminus\{ d_{k-1}e_{k-1}\}|\leq1$, by Lemma \ref{pathpartition9} there is a spanning path $P_{r_{k-1},r_{k-1}'}$ passing through $M_{k-1}\setminus\{d_{k-1}e_{k-1}\}$ in $Q[k-1]-\{d_{k-1},e_{k-1}\}$. Let $C=P_{c_0,d_0}+P_{e_0,f_0}+P_{c_1,f_1}+P_{r_{k-1},r_{k-1}'}+\{c_0c_1,f_0f_1,d_0d_{k-1},e_0e_{k-1},d_{k-1}e_{k-1},r_{0}r_{k-1},$
$r_{0}'r_{k-1}'\}-\{r_{0}r_{0}'\}$; see Figure 8.

If $j=k-2$, then $C$ is the desired Hamiltonian cycle in $Q_n^{k}$. Otherwise, $j<k-2$. When $M\cap E_d(k-2,k-1)=\emptyset$, since $|E(P_{r_{k-1},r_{k-1}'})|-|M_{k-1}|-|M_{k-2}|\geq k^{n-1}-3-2>1$, we may choose an edge $w_{k-1}w_{k-1}'\in E(P_{r_{k-1},r_{k-1}'})\setminus M_{k-1}$ such that $w_{k-2}w_{k-2}'\notin M_{k-2}$. When $M\cap E_d(k-2,k-1)\neq\emptyset$, now $|M\cap E_d(k-2,k-1)|=1$. Let $M\cap E_d(k-2,k-1)=\{w_{k-2}w_{k-1}\}$ and let $w_{k-1}'$ be a neighbor of $w_{k-1}$ on $P_{r_{k-1},r_{k-1}'}$. We also have $w_{k-1}w_{k-1}'\notin M_{k-1}$ and $w_{k-2}w_{k-2}'\notin M_{k-2}$. In the above two cases, since $|M[j+1,k-2]|\leq1$, by Lemma \ref{hamilpath2} there is a spanning path $P_{w_{k-2},w_{k-2}'}$ passing through $M[j+1,k-2]$ in $Q[j+1,k-2]$. Now $C+P_{w_{k-2},w_{k-2}'}+\{w_{k-2}w_{k-1},w_{k-2}'w_{k-1}'\}-\{w_{k-1}w_{k-1}'\}$ is the desired Hamiltonian cycle in $Q_n^{k}$; see Figure 8.

\begin{figure}[h]
\begin{center}
\includegraphics[scale=0.5]{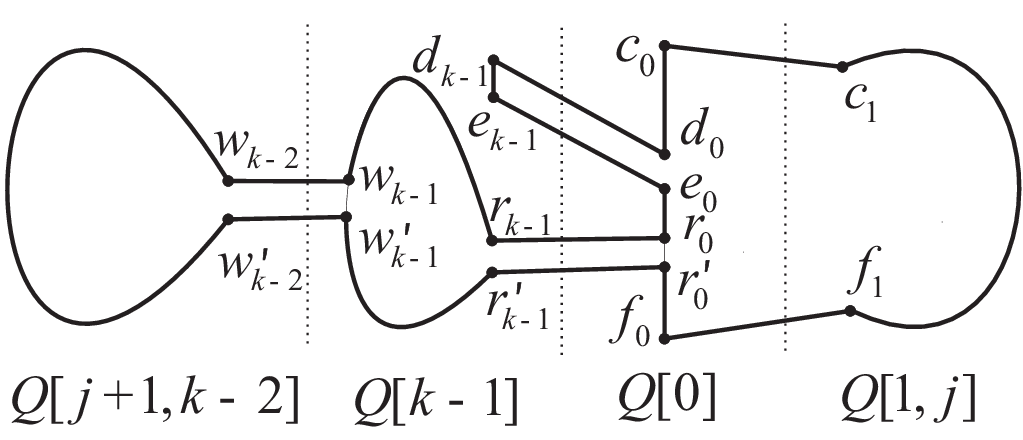}\\{Figure 8. Illustration for the proof of Subcase 1.1 in Claim 3.}
\end{center}
\end{figure}

\textbf{Subcase 1.2.} $M\cap E_d(i,i+1)\neq\emptyset$ for every $i\in\{1,\ldots,k-2\}$.

Since $\sum_{i=1}^{k-2}|M\cap E_d(i,i+1)|\leq|M[1,k-1]|\leq2$, then $k=4$ and $|M\cap E_d(1,2)|=|M\cap E_d(2,3)|=1$. Thus $M_{1}=M_{2}=M_3=\emptyset$. By Lemma \ref{pathpartition3} there is a spanning 2-path $P_{c_{1},f_{1}}+P_{d_{1},e_{1}}$ in $Q[1]$. Let $M\cap E_d(1,2)=\{a_1a_2\}$ and let $a_1'$ be a  neighbor of $a_1$ on $P_{c_{1},f_{1}}+P_{d_{1},e_{1}}$. By Lemma \ref{hamilpathone} there is a spanning path $P_{a_{2},a_2'}$ passing through $M[2,3]$ in $Q[2,3]$. Now $P_{c_0,d_0}+P_{e_0,f_0}+P_{c_{1},f_{1}}+P_{d_{1},e_{1}}+P_{a_{2},a_2'}+\{c_0c_{1},f_0f_{1},d_0d_{1},e_0e_{1},a_{1}a_{2},a_1'a_2'\}-\{a_{1}a_1'\}$ is the desired Hamiltonian cycle in $Q_n^{k}$.

\textbf{Case 2.} $|M\cap E_d(0,1)|+|M\cap E_d(k-1,0)|=1$. Now $|M[1,k-1]|\leq1$.

By symmetry we assume $M\cap E_d(0,1)=\{a_0a_1\}$ and $M\cap E_d(k-1,0)=\emptyset$. Since $k\geq4$, there exists an integer $j\in\{1,\ldots,k-2\}$ such that $M\cap E_d(j,j+1)=\emptyset$.

\textbf{Subcase 2.1.} $a_0\in\{c_0,d_0,e_0,f_0\}$, with loss of generality, we assume $a_0=c_0$.

Note that $c_1,f_1,d_{1},e_{1},d_{k-1},e_{k-1}$ are distinct and whenever $k$ is even, $p(c_1)\neq p(f_1)$, $p(e_1)\neq p(d_1)$ and $p(e_{k-1})\neq p(d_{k-1})$. If $d_{k-1}e_{k-1}\notin M_{k-1}$, since $|M[1,j]|\leq1$ and $|M[j+1,k-1]|\leq1$, by Lemma \ref{hamilpathone} there exist spanning paths $P_{c_{1},f_{1}}$ passing through $M[1,j]$ in $Q[1,j]$ and $P_{d_{k-1},e_{k-1}}$ passing through $M[j+1,k-1]$ in $Q[j+1,k-1]$, respectively. Hence, $P_{c_0,d_0}+P_{e_0,f_0}+P_{c_{1},f_{1}}+P_{d_{k-1},e_{k-1}}+\{c_0c_{1},f_0f_{1},d_0d_{k-1},e_0e_{k-1}\}$ is the desired Hamiltonian cycle in $Q_n^{k}$. If $d_{k-1}e_{k-1}\in M_{k-1}$, then $|M_{k-1}|=1$ and $|M[1,k-2]|=0$. By Lemma \ref{pathpartition7} there is a spanning 2-path $P_{c_{1},f_{1}}+P_{d_{1},e_{1}}$ passing through $M[1,k-1]$ in $Q[1,k-1]$. Now $P_{c_0,d_0}+P_{e_0,f_0}+P_{c_{1},f_{1}}+P_{d_{1},e_{1}}+\{c_0c_{1},f_0f_{1},d_0d_{1},e_0e_{1}\}$ is the desired Hamiltonian cycle in $Q_n^{k}$.

\textbf{Subcase 2.2.} $a_0\notin\{c_0,d_0,e_0,f_0\}$.

If $|M[1,j]|=0$, then since $|M_{k-1}|\leq|M[j+1,k-1]|\leq1$, we have $c_{k-1}f_{k-1}\notin M_{k-1}$ or $d_{k-1}e_{k-1}\notin M_{k-1}$. Without loss of generality, assume $d_{k-1}e_{k-1}\notin M_{k-1}$. Let $a_{0}'$ be a neighbor of $a_0$ on $P_{c_0,d_0}+P_{e_0,f_0}$ such that $a_{0}'\notin\{c_0,f_0\}$. By Lemma \ref{pathpartition3} there is a spanning 2-path $P_{c_{1},f_{1}}+P_{a_{1},a_{1}'}$ in $Q[1,j]$ and by Lemma \ref{hamilpathone} there is a spanning path $P_{d_{k-1},e_{k-1}}$ passing through $M[j+1,k-1]$ in $Q[j+1,k-1]$. Hence, $P_{c_0,d_0}+P_{e_0,f_0}+P_{c_{1},f_{1}}+P_{a_{1},a_{1}'}+P_{d_{k-1},e_{k-1}}+\{c_0c_{1},f_0f_{1},a_0a_{1},a_{0}'a_{1}',d_0d_{k-1},e_0e_{k-1}\}-\{a_0a_{0}'\}$ is the desired Hamiltonian cycle in $Q_n^{k}$. If $|M[1,j]|=1$, then $|M[j+1,k-1]|=0$. Let $a_{0}'$ be a neighbor of $a_0$ on $P_{c_0,d_0}+P_{e_0,f_0}$. There may happen the case $a_0'\in\{c_0,d_0,e_0,f_0\}$, but this does not affect our conclusion. By Lemma \ref{pathpartition3} there is a spanning 2-path $P_{c_{k-1},f_{k-1}}+P_{d_{k-1},e_{k-1}}$ in $Q[j+1,k-1]$ and by Lemma \ref{hamilpathone} there is a spanning path $P_{a_{1},a_{1}'}$ passing through $M[1,j]$ in $Q[1,j]$. Now $P_{c_0,d_0}+P_{e_0,f_0}+P_{a_{1},a_{1}'}+P_{c_{k-1},f_{k-1}}+P_{d_{k-1},e_{k-1}}+\{a_0a_{1},a_{0}'a_{1}',c_0c_{k-1},f_0f_{k-1},d_0d_{k-1},e_0e_{k-1}\}-\{a_0a_{0}'\}$ is the desired Hamiltonian cycle in $Q_n^{k}$.

\textbf{Case 3.} $|M\cap E_d(0,1)|+|M\cap E_d(k-1,0)|=2$. Now $|M[1,k-1]|=0$.

By symmetry, there are two cases to consider.

\textbf{Subcase 3.1.} $M\cap E_d(0,1)=\{a_0a_1,b_0b_{1}\}$ and $M\cap E_d(k-1,0)=\emptyset$.

If $\{c_0,d_0,e_0,f_0\}\cap\{a_0,b_0\}=\emptyset$, we can choose neighbors $a_{0}',b_{0}'$ of $a_{0},b_0$ on $P_{c_0,d_0}+P_{e_0,f_0}$ such that $a_0,a_{0}',b_0,b_{0}'$ are distinct. By Lemma \ref{pathpartition3} there exist spanning 2-paths $P_{a_1,a_{1}'}+P_{b_1,b_{1}'}$ in $Q[1,k-2]$ and $P_{c_{k-1},f_{k-1}}+P_{d_{k-1},e_{k-1}}$ in $Q[k-1]$. Now $P_{c_0,d_0}+P_{e_0,f_0}+P_{c_{k-1},f_{k-1}}+P_{d_{k-1},e_{k-1}}+P_{a_1,a_{1}'}+P_{b_1,b_{1}'}+\{c_0c_{k-1},f_0f_{k-1},d_0d_{k-1},e_0e_{k-1},$
$a_0a_{1},a_0'a_{1}',b_0b_{1},b_0'b_{1}'\}-\{a_0a_0',b_0b_0'\}$ is the desired Hamiltonian cycle in $Q_n^{k}$.

If $|\{c_0,d_0,e_0,f_0\}\cap\{a_0,b_0\}|=1$, without loss of generality, assume $a_0=c_0$ and $b_0\notin\{d_0,e_0,f_0\}$. Let $b_{0}'$ be a neighbor of $b_0$ on $P_{c_0,d_0}+P_{e_0,f_0}$ such that $b_{0}'\notin\{c_0,f_0\}$. By Lemma \ref{pathpartition3} there is a spanning 2-path $P_{c_{1},f_{1}}+P_{b_{1},b_{1}'}$ in $Q[1,k-2]$ and by Lemma \ref{hamilpath1} there is a spanning path $P_{d_{k-1},e_{k-1}}$ in $Q[k-1]$. Now $P_{c_0,d_0}+P_{e_0,f_0}+P_{c_{1},f_{1}}+P_{b_{1},b_{1}'}+P_{d_{k-1},e_{k-1}}+\{c_0c_{1},f_0f_{1},d_0d_{k-1},e_0e_{k-1},b_0b_{1},b_{0}'b_{1}'\}-\{b_0b_{0}'\}$ is the desired Hamiltonian cycle in $Q_n^{k}$.

If $|\{c_0,d_0,e_0,f_0\}\cap\{a_0,b_0\}|=2$, then by Lemma \ref{pathpartition3} there is a spanning 2-path $P_{c_{1},f_{1}}+P_{d_{1},e_{1}}$ in $Q[1,k-1]$. Hence, $P_{c_0,d_0}+P_{e_0,f_0}+P_{c_{1},f_{1}}+P_{d_{1},e_{1}}+\{c_0c_{1},f_0f_{1},d_0d_{1},e_0e_{1}\}$ is the desired Hamiltonian cycle in $Q_n^{k}$.

\textbf{Subcase 3.2.} $M\cap E_d(0,1)=\{a_0a_1\}$ and $M\cap E_d(k-1,0)=\{b_0b_{k-1}\}$.

If $\{c_0,d_0,e_0,f_0\}\cap\{a_0,b_0\}=\emptyset$, without loss of generality, assume $a_0\in V(P_{c_0,d_0})$. If $b_0\in V(P_{c_0,d_0})$, without loss of generality, assume $b_0$ is closer to $c_0$ than $a_{0}$ on $P_{c_0,d_0}$. Choose the neighbor $a_{0}'$ of $a_{0}$ on $P_{c_0,d_0}$ which is closer to $d_0$, and choose the neighbor $b_{0}'$ of $b_{0}$ on $P_{c_0,d_0}$ which is closer to $c_0$. If $b_0\notin V(P_{c_0,d_0})$, then we can choose neighbors $a_{0}'$ and $b_{0}'$ of $a_{0}$ and $b_0$ on $P_{c_0,d_0}$ and $P_{e_0,f_0}$, respectively, such that $a_{0}'\neq c_0$ and $b_{0}'\neq e_0$. In the above two cases, we both have $a_{0}a_{0}'\neq b_0b_{0}'$ and $a_{1},a_{1}',c_{1},f_{1},b_{k-1},b_{k-1}',d_{k-1},e_{k-1}$ are distinct. By Lemma \ref{pathpartition3} there exist spanning 2-paths $P_{a_{1},a_{1}'}+P_{c_{1},f_{1}}$ in $Q[1,k-2]$ and $P_{b_{k-1},b_{k-1}'}+P_{d_{k-1},e_{k-1}}$ in $Q[k-1]$, respectively. Now $P_{c_0,d_0}+P_{e_0,f_0}+P_{a_{1},a_{1}'}+P_{c_{1},f_{1}}+P_{b_{k-1},b_{k-1}'}+P_{d_{k-1},e_{k-1}}+\{a_0a_{1},a_0'a_{1}',c_0c_{1},f_0f_{1},b_0b_{k-1},b_0'b_{k-1}',d_0d_{k-1},e_0e_{k-1}\}$
$-\{a_0a_0',b_0b_0'\}$ is the desired Hamiltonian cycle in $Q_n^{k}$.

If $|\{c_0,d_0,e_0,f_0\}\cap\{a_0,b_0\}|\geq1$, without loss of generality, assume $a_0=c_0$. Let $b_{0}'$ be a neighbor of $b_0$ on $P_{c_0,d_0}+P_{e_0,f_0}$. By Lemma \ref{pathpartition3} there is a spanning 2-path $P_{c_{1},f_{1}}+P_{d_{1},e_{1}}$ in $Q[1,k-2]$ and by Lemma \ref{hamilpath1} there is a spanning path $P_{b_{k-1},b_{k-1}'}$ in $Q[k-1]$. Now $P_{c_0,d_0}+P_{e_0,f_0}+P_{c_{1},f_{1}}+P_{d_{1},e_{1}}+P_{b_{k-1},b_{k-1}'}+\{c_0c_{1},f_0f_{1},d_0d_{1},e_0e_{1},b_0b_{k-1},b_{0}'b_{k-1}'\}-\{b_0b_{0}'\}$ is the desired Hamiltonian cycle in $Q_n^{k}$.

Claim 3 is proved.

\vskip 0.2cm

\textbf{Claim 4.} \emph{Let $\{x_0y_0,m_0n_0,s_0t_0\}\subseteq M_0$ and $C_0$ be a Hamiltonian cycle containing $M_0\setminus\{x_0y_0,m_0n_0,s_0t_0\}$ in $Q[0]$. If $\{x_0y_0,m_0n_0,s_0t_0\}\cap E(C_0)=\emptyset$ and $|M\cap E_d(0,1)|+|M\cap E_d(k-1,0)|+|M[1,k-1]|\leq1$, then we can construct a Hamiltonian cycle containing $M$ in $Q_n^{k}$.}

\vskip 0.2cm

Choose clockwise neighbors $x_0',y_0',m_0',n_0',s_0',t_0'$ of $x_0,y_0,m_0,n_0,s_0,t_0$ on $C_{0}$. Note that $x_0',y_0',m_0',n_0',s_0',t_0'$ are distinct and $\{x_0x_0',y_0y_0',m_0m_0',n_0n_0',$ $s_0s_0',t_0t_0'\}\cap M=\emptyset$. Thus, $C_0-\{x_0x_0',y_0y_0',m_0m_0',n_0n_0',s_0s_0',t_0t_0'\}+\{x_0y_0,m_0n_0,s_0t_0\}$ is a spanning 3-path passing through $M_0$ in $Q[0]$, denoted by $P_{c_0,d_0}+P_{e_0,f_0}+P_{g_0,h_0}$. Then $\{c_0,d_0,e_0,f_0,g_0,h_0\}=\{x_0',y_0',m_0',n_0',s_0',t_0'\}$, which is balanced when $k$ is even. First, choose a vertex in $\{c_0,d_0,e_0,$ $f_0,g_0,h_0\}$ randomly, say $d_0$, then there exists at least one vertex in $\{e_0,f_0,g_0,h_0\}$, say $e_0$, satisfying $p(e_0)\neq p(d_0)$ when $k$ is even. So $\{c_0,f_0,g_0,h_0\}$ is balanced when $k$ is even. Thus there exists at least one vertex in $\{g_0,h_0\}$, say $g_0$, satisfying $p(g_0)\neq p(f_0)$ when $k$ is even; otherwise, $p(g_0)=p(h_0)=p(f_0)$, a contradiction. Therefore, $c_0,d_0,e_0,f_0,g_0,h_0$ are distinct, and whenever $k$ is even, $p(e_0)\neq p(d_0)$, $p(g_0)\neq p(f_0)$ and $p(c_0)\neq p(h_0)$. Since $|M\cap E_d(0,1)|+|M\cap E_d(k-1,0)|\leq1$, by symmetry we may assume $M\cap E_d(k-1,0)=\emptyset$.

If $|M\cap E_d(0,1)|=1$, then $|M[1,k-1]|=0$. Let $M\cap E_d(0,1)=\{a_0a_1\}$ and without loss of generality assume $a_0\in V(P_{c_0,d_0})$. Let $a_0'$ be a neighbor of $a_0$ on $P_{c_0,d_0}$. Note that $a_1,a_1',f_1,g_1$ are distinct. By Lemma \ref{pathpartition3} there exist spanning 2-paths $P_{a_1,a_1'}+P_{f_1,g_1}$ in $Q[1,k-2]$ and $P_{d_{k-1},e_{k-1}}+P_{c_{k-1},h_{k-1}}$ in $Q[k-1]$, respectively. Now $P_{c_0,d_0}+P_{e_0,f_0}+P_{g_0,h_0}+P_{a_1,a_1'}+P_{f_1,g_1}+P_{d_{k-1},e_{k-1}}+P_{c_{k-1},h_{k-1}}$ $+\{c_0c_{k-1},h_0h_{k-1},d_0d_{k-1},e_0e_{k-1},f_0f_1,g_0g_{1},a_0a_1,a_0'a_1'\}-\{a_0a_0'\}$ is the desired Hamiltonian cycle in $Q_n^{k}$.

If $M\cap E_d(0,1)=\emptyset$, then since $|M[1,k-1]|\leq1$ and $k\geq4$, by symmetry we may assume $M_{k-1}=M\cap E_d(k-2,k-1)=\emptyset$. Since $|M_{1}|\leq1$, there exists at least one pair of vertices in $\{\{c_1,h_{1}\},\{d_1,e_{1}\},\{f_1,g_{1}\}\}$, say $\{c_1,h_{1}\}$, satisfying $c_1h_{1}\notin M_1$. By Lemma \ref{pathpartition3} there is a spanning 2-path $P_{d_{k-1},e_{k-1}}+P_{f_{k-1},g_{k-1}}$ in $Q[k-1]$ and by Lemma \ref{hamilpathone} there is a spanning path $P_{c_1,h_1}$ passing through $M[1,k-2]$ in $Q[1,k-2]$. Now $P_{c_0,d_0}+P_{e_0,f_0}+P_{g_0,h_0}+P_{c_1,h_1}+P_{d_{k-1},e_{k-1}}+P_{f_{k-1},g_{k-1}}$ $+\{c_0c_1,h_0h_{1},d_0d_{k-1},e_0e_{k-1},f_0f_{k-1},g_0g_{k-1}\}$ is the desired Hamiltonian cycle in $Q_n^{k}$.

Claim 4 is proved.

If $|M_0|\leq4n-24=4(n-1)-20$, then by the induction hypothesis there is a Hamiltonian cycle $C_0$ containing $M_0$ in $Q[0]$; therefore the theorem holds by Claim 1.

If $|M_0|=4n-23$, then let $x_0y_0\in M_0$. By the induction hypothesis there is a Hamiltonian cycle $C_0$ containing $M_0\setminus\{x_0y_0\}$ in $Q[0]$. When $x_0y_0\in E(C_0)$, the theorem holds by Claim 1. When $x_0y_0\notin E(C_0)$, since $|M\cap E_d(0,1)|+|M\cap E_d(k-1,0)|+|M[1,k-1]|=|M|-|M_0|\leq3$, the theorem holds by Claim 2.

If $|M_0|=4n-22$, then let $\{x_0y_0,m_0n_0\}\subseteq M_0$. By the induction hypothesis there is a Hamiltonian cycle $C_0$ containing $M_0\setminus\{x_0y_0,m_0n_0\}$ in $Q[0]$. Since $|M\cap E_d(0,1)|+|M\cap E_d(k-1,0)|+|M[1,k-1]|\leq2$, by Claim 1 in case $\{x_0y_0,m_0n_0\}\subseteq E(C_0)$, and Claim 2 in case $|\{x_0y_0,m_0n_0\}\cap E(C_0)|=1$, and Claim 3 in case $\{x_0y_0,m_0n_0\}\cap E(C_0)=\emptyset$, the theorem holds.

If $|M_0|=4n-21$, then let $\{x_0y_0,m_0n_0,s_0t_0\}\subseteq M_0$. By the induction hypothesis there is a Hamiltonian cycle $C_0$ containing $M_0\setminus\{x_0y_0,m_0n_0,s_0t_0\}$ in $Q[0]$. Since $|M\cap E_d(0,1)|+|M\cap E_d(k-1,0)|+|M[1,k-1]|\leq1$, by Claim 1 in case $\{x_0y_0,m_0n_0,s_0t_0\}\subseteq E(C_0)$, and Claim 2 in case $|\{x_0y_0,m_0n_0,s_0t_0\}\cap E(C_0)|=2$, and Claim 3 in case $|\{x_0y_0,m_0n_0,s_0t_0\}|\cap E(C_0)=1$, and Claim 4 in case $\{x_0y_0,m_0n_0,s_0t_0\}\cap E(C_0)=\emptyset$, the theorem holds.

If $|M_0|=4n-20$, then $M\cap E_d(0,1)=M\cap E_d(k-1,0)=M[1,k-1]=\emptyset$. Let $\{x_0y_0,m_0n_0,s_0t_0,$
$w_0z_0\}\subseteq M_0$. By the induction hypothesis there is a Hamiltonian cycle $C_0$ containing $M_0\setminus\{x_0y_0,m_0n_0,s_0t_0,w_0z_0\}$ in $Q[0]$. If $\{x_0y_0,m_0n_0,s_0t_0,w_0z_0\}\cap E(C_{0})\neq\emptyset$, then the conclusion holds by Claim 1-4. If $\{x_0y_0,m_0n_0,s_0t_0,w_0z_0\}\cap E(C_{0})=\emptyset$, then choose clockwise neighbors $x_0',y_0',m_0',n_0',s_0',t_0',w_0',z_0'$ of $x_0,y_0,m_0,n_0,s_0,t_0,w_0,z_0$ on $C_{0}$, respectively. Note that $x_0',y_0',m_0',n_0',s_0',t_0',w_0',z_0'$ are distinct, and $\{x_0x_0',y_0y_0',m_0m_0',$
$n_0n_0',s_0s_0',t_0t_0',w_0w_0',z_0z_0'\}\cap M=\emptyset$, and whenever $k$ is even, $p(x_0')\neq p(y_0')$, $p(m_0')\neq p(n_0')$, $p(s_0')\neq p(t_0')$ and $p(w_0')\neq p(z_0')$.

Note that $C_0-\{x_0x_0',y_0y_0',m_0m_0',n_0n_0',s_0s_0',t_0t_0',w_0w_0',z_0z_0'\}+\{x_0y_0,m_0n_0,s_0t_0,w_0z_0\}$ is a spanning 4-path passing through $M$ in $Q[0]$, denoted by $P_{a_0,b_0}+P_{c_0,d_0}+P_{e_0,f_0}+P_{g_0,h_0}$. Then $\{a_0,b_0,c_0,d_0,e_0,f_0,g_0,h_0\}$
$=\{x_0',y_0',m_0',n_0',s_0',t_0',w_0',z_0'\}$, which is balanced when $k$ is even. In the following description, whenever the parities of vertices are mentioned, it is implicitly assumed to be the case where $k$ is even. First, choose a vertex in $\{a_0,b_0,c_0,d_0,e_0,f_0,g_0,h_0\}$ randomly, say $b_0$, then there exists at least one vertex in $\{c_0,d_0,e_0,f_0,g_0,h_0\}$, say $c_0$, satisfying $p(c_0)\neq p(b_0)$. Then $\{a_0,d_0,e_0,f_0,g_0,h_0\}$ is balanced. So there exists at least one vertex in $\{e_0,f_0,g_0,h_0\}$, say $e_0$, satisfying $p(e_0)\neq p(d_0)$. Then $\{a_0,f_0,g_0,h_0\}$ is balanced. Hence, there exists at least one vertex in $\{g_0,h_0\}$, say $g_0$, satisfying $p(g_0)\neq p(f_0)$. Thus, $p(a_0)\neq p(h_0)$. Now $b_1,c_1,d_1,e_1,f_{k-1},g_{k-1},a_{k-1},h_{k-1}$ are distinct, and whenever $k$ is even, $p(b_1)\neq p(c_1)$, $p(d_1)\neq p(e_1)$, $p(f_{k-1})\neq p(g_{k-1})$ and $p(a_{k-1})\neq p(h_{k-1})$. By Lemma \ref{pathpartition3} there exist spanning 2-paths $P_{b_1,c_1}+P_{d_1,e_1}$ in $Q[1,k-2]$ and $P_{f_{k-1},g_{k-1}}+P_{a_{k-1},h_{k-1}}$ in $Q[k-1]$, respectively. Now $P_{a_0,b_0}+P_{c_0,d_0}+P_{e_0,f_0}+P_{g_0,h_0}+P_{b_1,c_1}+P_{d_1,e_1}+P_{f_{k-1},g_{k-1}}+P_{a_{k-1},h_{k-1}}+$
$\{b_0b_1,c_0c_1,d_0d_1,e_0e_1,f_0f_{k-1},g_0g_{k-1},a_0a_{k-1},h_0h_{k-1}\}$ is the desired Hamiltonian cycle in $Q_n^{k}$.

The proof of Theorem \ref{mosttheorem} is complete.

\vskip 0.4cm

\noindent{\large\bf Acknowledgements}





\vskip 0.2cm

This work is supported by National Natural Science Foundation of China (grant no. 12061047), and supported by Natural Science Foundation of Jiangxi Province (nos. 20212BAB201027 and 20192BAB211002).

\vskip 0.4cm

\noindent{\large\bf Statements and Declarations}

\vskip 0.2cm

\noindent{\textbf{Competing Interests}}

\vskip 0.2cm

The authors declare that they have no known competing financial interests or personal relationships that could have appeared to influence the work reported in this paper.


\end{document}